\documentclass[11pt]{article}

\usepackage{amsfonts,amsmath}
\usepackage{latexsym}
\usepackage{bbm}
\usepackage{todonotes}
\usepackage{changes}
\definechangesauthor[name={You}, color=blue]{Y}

\usepackage{mathtools}

\usepackage{amsmath}
\usepackage{amssymb}
\usepackage{amsthm}
\usepackage{mathrsfs}
\usepackage{amssymb,url}

\author{K\'aroly J. B\"or\"oczky\footnote{Alfr\'ed R\'enyi Institute of Mathematics, Realtanoda u. 13-15, H-1053 Budapest, Hungary, boroczky.karoly.j@renyi.hu}, 
Shibing Chen\footnote{School of Mathematical Sciences, University of Science and Technology of China, 230026 Hefei, China, chenshib@ustc.edu.cn}, 
Weiru Liu\footnote{School of Mathematics and Statistics, Central China Normal University,  
430079 Wuhan, China, lwr1997@ccnu.edu.cn}, 
Christos Saroglou\footnote{Department of Mathematics, University of Ioannina, Greece, csaroglou@uoi.gr}}
\title{Compactness of the $L_p$ dual Minkowski problem in $\R^3$}

\newcommand{\R}{\mathbb{R}}

\newcommand{\HH}{\mathcal{H}}

\newtheorem{lemma}{LEMMA}[section]
\newtheorem{theo}[lemma]{THEOREM}

\newtheorem{coro}[lemma]{COROLLARY}

\newtheorem{prop}[lemma]{PROPOSITION}

\newtheorem{problem}[lemma]{PROBLEM}

\begin{document}

\maketitle

\begin{abstract}
We prove the $C^0$ estimate for the $L_p$ $q$th dual Minkowski problem on $S^2$ under fairly general conditions; namely, when $p\in[0,1)$ and $q>2+p$, and the $L_p$ $q$th dual curvarture is bounded and bounded away from zero. We note that it is known that the analogous  $C^0$ estimate does not hold if  $p<-1$ and $q=3$. As a corollary of our $C^0$ estimate, we deduce the uniqueness of the solution of the near isotropic  $q$th $L_p$ dual Minkowski problem on $S^2$ if $q$ is close to $3$ and the $q$th $L_p$ dual curvature is H\"older close to be the constant one function.
\end{abstract}

\noindent {\bf MSC classes:} 35J96 (52A20)

\section{Introduction}
\label{secIntroduction}

The $C^0$ estimates are fundamental tools in order to obtain existence or uniqueness results in the theory of Monge-Amp\`ere equations. In this paper, we concentrate on the recently introduced $L_p$ dual Minkowski problem (cf. Huang, Lutwak, Yang, Zhang \cite{HLYZ16} and Lutwak, Yang, Zhang \cite{LYZ18}), and the corresponding Monge-Amp\`ere equation \eqref{Lp-dualMink-Monge-Ampere-intro}:
For a non-negative function $f\in L^1(S^{n-1})$, $q>0$ and $p\in\R$, the  Monge-Amp\`ere equation on $S^{n-1}$ corresponding to the $L_p$ $q$th dual Minkowski problem  is
\begin{equation}
\label{Lp-dualMink-Monge-Ampere-intro}
\left(\|\nabla u\|^2+u^2\right)^{\frac{q-n}2}\cdot u^{1-p}\det\left(\nabla^2 u+u\,I\right)=f,
\end{equation}
where $\nabla u$ and  $\nabla^2 u$ are the spherical gradient and the Hessian of the function $u$ on $S^{n-1}$ with respect to a moving orthonormal frame, and $I$ denotes the identity matrix of suitable size.

Let us introduce the notions and notations used in the Alexandrov version of the Monge-Amp\`ere equation \eqref{Lp-dualMink-Monge-Ampere-intro} (see also Section~\ref{sec-basics}). We write $\HH^k$ to denote the $k$-dimensional Hausdorff measure normalized in a way such that it coincides with the Lebesgue measure on $k$-dimensional affine subspaces, and to simplify formulas, we also use the notation
$|X|=\mathcal{H}^n(X)$ for a measurable $X\subset\R^n$. We note that for any measure in this paper, the measure of the emptyset is $0$.
  For a compact convex set $K\subset\R^n$, 
 the support function 
$h_K(w)=\max_{x\in K}\langle w,x\rangle$ for $w\in \R^n$ is convex and $1$-homogeneous, and the surface area measure on $S^{n-1}$ is denoted by $S_K$. We note that if ${\rm dim}\,K\leq n-2$, then $S_K$ is the zero measure, and if $w\in S^{n-1}$ and $K\subset w^\bot$ is an $(n-1)$-dimensional compact convex set, then $S_K$ is concentrated on $\{w,-w\}$ in a way such that $S_K(\{w\})=S_K(\{-w\})=\mathcal{H}^{n-1}(K)$ (see Schneider \cite{Sch14}). 

We call a compact convex set in $\R^n$ with non-empty interior a convex body, and the family of convex bodies in $\R^n$ containing the origin $O$ is denoted by $\mathcal{K}^n_o$. For $K\in\mathcal{K}^n_o$, let $\partial'K$ denote the subset of the boundary  of $K$ such that there exists a unique exterior unit normal vector
$\nu_K(x)$ at any point $x\in \partial'K$.
It is well-known that $\HH^{n-1}(\partial K\setminus\partial'K)=0$ and $\partial'K$ is a Borel set  (see Schneider \cite{Sch14}).  The function $\nu_K:\partial'K\to S^{n-1}$ is the spherical Gauss map
that is continuous on $\partial'K$.
Now the surface area measure $S_K$ of $K\in\mathcal{K}^n_o$ is a Borel measure on $S^{n-1}$ 
satisfying that $S_K(\eta)=\HH^{n-1}(\nu_K^{-1}(\eta))$
 for any Borel set $\eta\subset S^{n-1}$. In particular, if $\partial K$ is $C^2_+$, then 
\begin{equation}
\label{SKMonge}
dS_K=\det(\nabla^2 u+u\,I)d\HH^{n-1}
\end{equation}
where $u=h_K|_{S^{n-1}}$.

For $q>0$ and $K\in \mathcal{K}^n_o$, to define the $q$th dual curvature measure $\widetilde{C}_{q,K}$ on $S^{n-1}$  initated by Huang, Lutwak, Yang, Zhang \cite{HLYZ16} (see also  B\"or\"oczky, Fodor \cite{BoF19} for the case when $O\in\partial K$), first we consider the radial function $\varrho_K$   defined by
$$
\varrho_K(w)=\max\{t\geq 0:\,tw\in K\}
$$
for $w\in S^{n-1}$. 
For $\omega\subset S^{n-1}$, let
$$
\alpha^*_K(\omega)=\{w\in S^{n-1}:\nu_K(\varrho_K(w)\cdot w)\cap\omega\neq \emptyset \},
$$
which is $\mathcal{H}^{n-1}$ measurable if $\omega$ is measurable.
Now for a $\mathcal{H}^{n-1}$ measurable $\omega\subset S^{n-1}$, its $q$th dual curvature measure is
$$
\widetilde{C}_{q,K}(\omega)=\int_{\alpha^*_K(\omega)}\varrho_K^{q}\,d\mathcal{H}^{n-1}.
$$
Since $\widetilde{C}_{q,K}(\{v\in S^{n-1}:\,h_K(v)=0\})=0$ for given $q>0$ and $K\in\mathcal{K}_o^n$ (cf. \eqref{Cq-nuo-zero}), it is possible to extend the definition of Lutwak, Yang, Zhang \cite{LYZ18} of the $L_p$ $q$th dual  curvature measure on $S^{n-1}$ as
$$
d\widetilde{C}_{p,q,K}= u^{-p}\,d\widetilde{C}_{q,K}
$$
for $u=h_K|_{S^{n-1}}$ and  $p\in\R$ (cf. B\"or\"oczky, Fodor \cite{BoF19}). It is  a finite measure on $S^{n-1}$ if $o\in{\rm int}\,K$ or $p\leq 0$ or $0<p<\min\{q,1\}$ (see B\"or\"oczky, Chen, Liu, Saroglou  \cite{BCLS} for the latter property). As special cases, if $q>0$ and $p\in\R$, then (see Lutwak, Yang, Zhang \cite{LYZ18})
\begin{align*}
\widetilde{C}_{0,q,K}=&\widetilde{C}_{q,K},\\
\widetilde{C}_{p,n,K}=&h_K^{1-p}\,dS_K=S_{p,K}\mbox{ \ is Lutwak's $L_p$ surface area measure,}\\
\widetilde{C}_{0,n,K}=&\widetilde{C}_{n,K}=S_{0,K}=nV_K
\end{align*}
where $V_K$ is Firey's (cf. \cite{Fir74}) cone volume measure that nicely intertwines with linear maps (cf. Section~\ref{sec-basics}).  In particular, Lutwak's $L_p$ Minkowski problem (the case $q=n$ of \eqref{Lp-dualMink-Monge-Ampere-intro})  is
\begin{equation}
\label{LpMonge}
 u^{1-p}\det\left(\nabla^2 u+u\,I\right)=f
\end{equation}
on $S^{n-1}$ where the case $p=0$ is Firey's (cf. \cite{Fir74}) Logarithmic Minkowski problem about the existence of the cone volume measure, and the case $p=1$ is the classical Minkowski problem solved by the work of
Minkowski \cite{Min03,Min11}, Aleksandrov \cite{Ale38a,Ale96},
 Nirenberg \cite{Nir53}, Cheng, Yau \cite{ChY76}, Pogorelov \cite{Pog78}, Caffarelli \cite{Caf90a,Caf90b} stretching throughout the 20th century. For some existence results about the $L_p$-Minkowski problem without evenness condition for $p\neq 1$, 
see for example Chou,  Wang \cite{ChW06}, Chen, Li, Zhu \cite{CLZ17,CLZ19}, 
Bianchi, B\"or\"oczky, Colesanti, Yang \cite{BBCY19} and
Guang, Li, Wang \cite{LGWa}.

For $q>0$ and $p\in\R$, a $K\in\mathcal{K}_o^n$ is a solution of the $L_p$ dual Minkowski problem \eqref{Lp-dualMink-Monge-Ampere-intro} in Alexandrov's sense (in the sense of measure) if
\begin{equation}
\label{Lp-dualMink-Monge-Ampere-measure-intro}
d\widetilde{C}_{p,q,K}=f\,d\mathcal{H}^{n-1}.
\end{equation}
Now $C^0$ estimates are cornerstones of existence and uniqueness results about the $L_p$ dual Minkowski problem or similar Monge-Amp\`ere equations, either in the case of the variational or the flow method, see for example 
Andrews \cite{And99}, B\"or\"oczky,  Lutwak, Yang, Zhang \cite{BLYZ13},  Brendle, Choi, Daskalopoulos \cite{BCD17}, Chou, Wang \cite{ChW06}, Li, Sheng, Wang \cite{LSW20}. Our main theme is to discuss the following problem.

\begin{problem}\label{problem-C0-estimate}
For $n\geq 2$, $p\in(-1,1)$, $q>n-1$ and $\lambda>1$, does there exist a constant $C=C(\lambda,n,p,q)>1$ such that if 
$$
\lambda^{-1}\mathcal{H}^{n-1}\leq \widetilde{C}_{p,q,K}\leq\lambda\mathcal{H}^{n-1}
$$
holds on $S^{n-1}$ for a convex body $K\in \mathcal{K}^n_o$,
then 
$$
\sup_{x\in S^{n-1}}h_K(x)\leq C\mbox{ and }|K|\geq C^{-1}.
$$
\end{problem}
\noindent{\bf Remark.} For any $p\in(-n,-1)$ and $n\geq 3$, Jian, Lu, Wang \cite{JLW15} (see (2.12), (2.13) and (2.14) in \cite{JLW15}) prove the existence of a $\lambda>1$ depending on $p$ and $n$ such that for any $\varepsilon>0$, there exists origin symmetric convex body $K_\varepsilon\subset\R^n$ with axial rotational symmetry and $C^\infty_+$ boundary satisfying that
$$
\frac1{\lambda}\cdot\mathcal{H}^{n-1}\leq S_{p,K_\varepsilon}\leq \lambda\cdot \mathcal{H}^{n-1},
$$
and  $|K_\varepsilon|<\varepsilon$. In particular, the lower bound $-1$ for $p$ in Problem~\ref{problem-C0-estimate} is necessary.

Our main result essentially solves Problem~\ref{problem-C0-estimate} if $n=3$. 

\begin{theo}
\label{n3-ppos--C0-estimate}
For $p\in[0,1)$, $q>2+p$ and $\lambda>1$,  there exists a constant $C=C(\lambda,p,q)>1$ such that if 
$$
\lambda^{-1}\mathcal{H}^{2}\leq \widetilde{C}_{p,q,K}\leq\lambda\mathcal{H}^{2}
$$
holds on $S^{2}$ for a convex body $K\in \mathcal{K}^3_o$,
then 
$$
\sup_{x\in S^{2}}h_K(x)\leq C\mbox{ and }|K|\geq C^{-1}.
$$
\end{theo}

We note that if $n=3,4$ and $q$ is very close to $n$, then B\"or\"oczky, Chen, Liu, Saroglou  \cite{BCLS} verifies Problem~\ref{problem-C0-estimate}. 

Let us discuss an application of Theorem~\ref{n3-ppos--C0-estimate} to uniqueness in the near isotropic $L_p$ dual Minkowski problem. In his classical paper about the Gauss curvature flow and ``worn stones", Firey \cite{Fir74}  asked in 1974 whether the solution of the $L_p$ Minkowski problem 
\eqref{LpMonge} is unique in the case when $n\geq 3$, $p=0$ and $f$ is the constant one function (the last property is usually referred to as the isotropic case). Now the solution of 
\eqref{LpMonge} is known to be unique for any positive bounded function $f$ if $p>1$
according to Hug, Lutwak, Yang, Zhang \cite{HLYZ05} and Chou, Wang \cite{ChW06}, but uniqueness may fail if $p<1$ (see Chen, Li, Zhu \cite{CLZ17} if $p\in(0,1)$, Chen, Li, Zhu \cite{CLZ19} if $p=0$,   Li, Liu, Lu \cite{LLL22} if $p<0$, and also E. Milman \cite{Mil24} for a systematic study). In addition, if $n\geq 3$ and $q=n$ and $p=-n$, then even in the isotropic case when $f$ is the constant one function, ellipsoids of given volume show that the solution of 
\eqref{Lp-dualMink-Monge-Ampere-intro} is not unique according to
 the earlier papers Calabi \cite{Cal58}, Pogorelov \cite{Pog72}, Cheng, Yau \cite{ChY86}, and the novel approaches Crasta, Fragal\'a \cite{CrF23}, Ivaki, E. Milman \cite{IvM23} and Saroglou \cite{Sar22}. On the other hand, the celebrated paper Brendle, Choi, Daskalopoulos \cite{BCD17} established the uniqueness of the solution of \eqref{Lp-dualMink-Monge-Ampere-intro} in the isotropic case if $n\geq 3$, $p\in(-n,1)$ and $q=n$ (see also Saroglou \cite{Sar22}).

Using the method in B\"or\"oczky, Chen, Liu, Saroglou  \cite{BCLS}, one obtains the following uniqueness result based on Theorem~\ref{n3-ppos--C0-estimate}.

\begin{coro}
\label{local-uniqueness-Cpq}
For $\alpha\in(0,1)$ and $p\in[0,1)$, there exists a constant $\varepsilon_0\in(0,1)$ that depends only on $\alpha$ and $p$ such that if $|q-3|<\varepsilon_0$ and 
$f\in C^{\alpha}(S^{2})$ satisfies $\|f-1\|_{C^{\alpha}(S^{2})}<\varepsilon_0$, then the equation
$$
d\widetilde{C}_{p,q,K}=fd\mathcal{H}^{2}
$$
has a unique solution in ${\cal K}^3_o$ in the sense of  Alexandrov, and $u=h_K|_{S^{2}}$ is a positive $C^{2,\alpha}$ solution of \eqref{Lp-dualMink-Monge-Ampere-intro}.
\end{coro}

Actually, B\"or\"oczky, Chen, Liu, Saroglou  \cite{BCLS} prove 
Corollary~\ref{local-uniqueness-Cpq}, and its analogues for any $n\geq 3$.
  
We note that question of uniqueness of the solution has been discussed in various versions of the Minkowski problem.
The Gaussian surface area measure of a $K\in\mathcal{K}^n_o$ is defined by
Huang, Xi and Zhao \cite{HXZ21}, whose results are extended  by Feng, Liu, Xu \cite{FLX23}, Liu \cite{Liu22} and Feng, Hu, Xu \cite{FHX23}. The fact that only balls have the property that the density of the Gaussian surface area measure is constant 
is proved by Chen, Hu, Liu, Zhao \cite{CHLZ} for convex domains in $\R^2$, and in the even case for $n\geq 3$
by \cite{CHLZ} and Ivaki, Milman \cite{IvM23}. The intensively investigated $L_p$-Minkowski conjecture
stated by B\"or\"oczky, Lutwak, Yang,  Zhang \cite{BLYZ12}
claims the uniqueness of the even solution of \eqref{LpMonge} for even positive $f$, 
see for example
B\"or\"oczky, Kalantzopoulos \cite{BoK22},
Chen, Huang, Li,  Liu \cite{CHLL20},
Colesanti,  Livshyts, Marsiglietti \cite{CLM17},
Colesanti,  Livshyts \cite{CoL20},
Ivaki, E. Milman \cite{IvM23b}, 
Kolesnikov \cite{Kol20},
Kolesnikov, Livshyts \cite{KoL},
Kolesnikov, Milman \cite{KoM22},
Livshyts, Marsiglietti, Nayar, Zvavitch \cite{LMNZ20},
Milman \cite{Mil24,Mil25}, 
Saroglou \cite{Sar15,Sar16},
Stancu \cite{Stancu,Stancu1,Sta22},
van Handel \cite{vHa} 
for partial results, and B\"or\"oczky \cite{Bor23} for a survey of the subject.

For the $L_p$ $q$th dual Minkowski problem
due to Lutwak, Yang, Zhang \cite{LYZ18} (see Huang, Lutwak, Yang, Zhang \cite{HLYZ16}
for the original $q$th dual Minkowski problem), Chen, Li \cite{ChL21} and
Lu, Pu \cite{LuP21} solve essentially the case $p>0$, and Huang, Zhao \cite{HuZ18} and Guang, Li, Wang \cite{LGW23} discuss the case $p<0$ (see Gardner, Hug, Weil, Xing, Ye \cite{GHWXY19,GHXY20} for Orlicz versions of some of these results using the variational method, and Li, Sheng, Wang \cite{LSW20} for an approach via the flow method). Uniqueness of the solution of the $L_p$ $q$th  dual Minkowski problem is thoroughly investigated by Li, Liu, Lu \cite{LLL22}. The case when $n=2$ and $f$ is a constant function has been completely clarified by
Li,  Wan \cite{LiW}. If the $L_p$ $q$th dual curvature is constant, then Ivaki, Milman \cite{IvM23} shows uniqueness of the even solution when $p>-n$ and $q\leq n$.\\

Concerning the structure of the paper, Section~\ref{sec-basics} summarizes the basic properties of $L_p$ dual curvature measures that we need, and Section~\ref{secRegularity} describe the corresponding regularity properties. A ``basic estimate" frequently used in the argument for Theorem~\ref{n3-ppos--C0-estimate} is verified in Section~\ref{sec-n3-ppos--C0-estimate-basic}, 
and finally, Theorem~\ref{n3-ppos--C0-estimate} is proved in Section~\ref{sec-n3-ppos--C0-estimate}.

\section{Some basic notation and notions}
\label{sec-basics}

This section discusses the basic notions and notations used throughout the paper, extending the discussion at the beginning of Section~\ref{secIntroduction}. For notions in convexity, see Schneider \cite{Sch14}. 
   As usual,  ${\rm int}\,X$ stands for the interior of an $X\subset \R^n$.
For a compact subset $K\subset \R^n$, we write ${\rm relint}\,K$ and $\partial K$ to denote its relative interior and relative boundary with respect to the topology of the linear hull of $K$ (that is the usual notion of boundary if ${\rm int}\,K\neq \emptyset$). 

For a convex body $K\in\mathcal{K}^n_o$,
 if $v\in S^{n-1}$ is an exterior unit normal at $x=\varrho_K(w)\cdot w\in\partial K$, then
\begin{equation}
\label{hK-rhoK}
h_K(v)=\langle \varrho_K(w)\cdot w,v\rangle\leq \varrho_K(w),
\end{equation}
and hence if in addition, $h_K$ is differentiable at $v$, and $u=h_K|_{S^{n-1}}$,  then 
\begin{align}
\label{DhK-u-x}
Dh_K(v)=&\nabla u(v)+u(v)\cdot v=x;\\
\label{DhK-length-rhoK}
\|Dh_K(v)\|=&\sqrt{\|\nabla u\|^2+u^2}=\varrho_K(w)=\|x\|;\\
\label{DhK-rhoK}
h_K(v)=&u(v)\leq  \|Dh_K(v)\|.
\end{align}
	
For $K\in \mathcal{K}^n_o$, another fundamental measure is the cone volume measure $dV_K=\frac1n\,h_K\,dS_K$. To provide a geometric interpretation of the cone volume measure,
	we consider the Borel measure $\widetilde{V}_K$ on $\partial K$, namely, for measurable $Z\subset\partial K$, we define
	\begin{equation}
		\label{tildeVKZ-def} 
		\widetilde{V}_K(Z)=\frac1n\int_{Z\cap \partial' K}\langle x,\nu_K(x)\rangle\,dx=\frac1n\int_{Z\cap \partial' K}h_K(\nu_K(x))\,dx.
	\end{equation}
	In particular,  $\widetilde{V}_K$ is inspired by the pullback of the cone volume measure to $\partial K$; namely, for Borel $\omega\subset S^{n-1}$, \eqref{tildeVKZ-def} yields that
	\begin{equation}
		\label{tildeVKZ-VK}
	V_K(\omega)=	\widetilde{V}_K(\{x\in \partial' K: \nu_K(x)\in \omega\}).
	\end{equation}

	Next we provide an interpretation of $\widetilde{V}_K$ in terms of the radial function $\varrho_K$ of a $K\in \mathcal{K}^n_o$ on $S^{n-1}$. We consider the open convex cone $\Sigma_K=\{(0,\infty)x:x\in{\rm int}\,K\}$, and hence $\Sigma_K=\R^n$
	if $o\in{\rm int}K$. 
	Since
	$\{x\in\partial' K:\langle \nu_K(x),x\rangle=0\}= \partial \Sigma_K\cap \partial' K$,
	we have $\widetilde{V}_K(\partial K\cap \partial \Sigma_K)=0$ if $\partial \Sigma_K\neq \emptyset$, the paper 
	Huang, Lutwak, Yang, Zhang \cite{HLYZ16} working on the case  when $o\in{\rm int}\,K$ yields that if $Z\subset \partial K$ is measurable, then 
\begin{equation}
\label{tildeVKZ-def2}
\widetilde{V}_K(Z)=\frac1n\int_{\pi_{S^{n-1}}(Z)}\varrho_K^n\,\mathcal{H}^{n-1}
=\left|\bigcup\{{\rm conv}\{o,x\}:\,x\in Z\}\right|,
\end{equation} 
where $\pi_{S^{n-1}}(x)=\frac{x}{||x||}$ for $x\neq O$.
For any measurable $\varphi:\,\partial K\to[0,\infty)$, we deduce from \eqref{tildeVKZ-def}  that 
\begin{equation}
\label{tildeVK-varphi-rhoK}
\int_{\partial K}\varphi(x)\,d\widetilde{V}_K(x)=\int_{\partial' K}\varphi(x)\langle \nu_K(x),x\rangle\,d\mathcal{H}^{n-1}(x).
\end{equation}
We note that if $\omega\subset S^{n-1}$, $Z=\nu_K^{-1}(\omega)\subset \partial K$ and $\Phi\in{\rm GL}(n)$ then $\langle \Phi x,\Phi^{-t}v\rangle=\langle x,v\rangle$ for $x,v\in\R^n$ yields that 
	\begin{equation}
		\label{normals-linear-image}
		\nu_{\Phi K}^{-1}\left(\Phi^{-t}_*\omega\right)=\Phi Z\mbox{ \ for \ }\Phi^{-t}_*\omega=\left\{\frac{\Phi^{-t}v}{\|\Phi^{-t}v\|}:\,v\in\omega\right\}.
	\end{equation}
	Now $\widetilde{V}_K$ nicely intertwines with linear transformations; namely, if $\Phi\in{\rm GL}(n)$ and $Z\subset\partial K$ is $\mathcal{H}^{n-1}$ measurable, then \eqref{tildeVKZ-def2} implies that
\begin{equation}
\label{tildeVK-linear-inv}
\widetilde{V}_{\Phi\,K}(\Phi\,Z)=
|\det\Phi|\cdot \widetilde{V}_K(Z),
\end{equation}
and hence for any measurable $\omega\subset S^{n-1}$, we deduce from \eqref{normals-linear-image} that
\begin{equation}
\label{VK-linear-inv}
V_{\Phi\,K}\left(\Phi^{-t}_*\omega\right)=
|\det\Phi|\cdot V_K(\omega),
\end{equation}
thus if $\varphi:S^{n-1}\to [0,\infty)$ is continuous, then
\begin{equation}
\label{VK-linear-inv-phi}
\int_{S^{n-1}}\varphi\,dV_{K}=|{\rm det} \Phi|^{-1} \int_{S^{n-1}}\varphi\circ \Phi^{t}_*\,dV_{\Phi K}.
\end{equation}
It follows from \eqref{tildeVKZ-def2} that
\begin{equation}
\label{VK-sphere-volume}
V_K(S^{n-1})=\widetilde{V}_K(\partial K)=|K|.
\end{equation}

For the $q$th dual curvature measure  of a $K\in \mathcal{K}^n_o$,  if $o\in\partial K$ and $q>0$, then readily
\begin{equation}
\label{Cq-nuo-zero}
\widetilde{C}_{q,K}(\nu_K(o))=\widetilde{C}_{q,K}\Big(\{v\in S^{n-1}:\,h_K(v)=0\}\Big)=0.
\end{equation}
Thus using \eqref{tildeVK-varphi-rhoK}, the dual curvature measure of any $K\in \mathcal{K}^n_o$ can be written as a surface integral (cf. Huang, Lutwak, Yang, Zhang \cite{HLYZ16})
\begin{align}
\label{dual-curvature-cone-surface}
\widetilde{C}_{q,K}(\omega)=&
\int_{\nu_K^{-1}(\omega)\cap\partial' K}\|x\|^{q-n}\langle x,\nu_K(x)\rangle\,d\mathcal{H}^{n-1}(x)\\
\label{dual-curvature-cone-vol}
=&n\int_{\nu_K^{-1}(\omega)\cap\partial' K}\|x\|^{q-n}\,d\widetilde{V}_K(x)
\end{align}
for a measurable $\omega\subset S^{n-1}$. According to \eqref{dual-curvature-cone-surface}, if $\varphi:S^{n-1}\to[0,\infty)$ is  measurable, then
\begin{align} 
\label{dual-curvature-phi-cone-on-boundary}
\int_{S^{n-1}}\varphi\,d\widetilde{C}_{q,K}=&
\int_{\partial' K}\varphi(\nu_K(x))\cdot \|x\|^{q-n}\langle x,\nu_K(x)\rangle\,d\mathcal{H}^{n-1}(x)\\
\label{dual-curvature-phi-cone-radial}
=&\int_{S^{n-1}}\varphi \cdot(\varrho_K\circ \alpha^*_K)^{q}\,d\mathcal{H}^{n-1}.
\end{align}
It follows from \eqref{dual-curvature-phi-cone-on-boundary}  that if $h_K$ is differentiable at each point of a measurable $\omega\subset S^{n-1}$ and $\varphi:\omega\to[0,\infty)$ is measurable, then
\begin{align}
\label{dual-curvature-phi-cone-surface}
\int_{\omega}\varphi\,d\widetilde{C}_{q,K}=&
\int_{\omega}\varphi \cdot \|Dh_K\|^{q-n}h_K\,dS_K\\
\label{dual-curvature-phi-cone-vol}
=& n\int_{\omega}\varphi \cdot\|Dh_K\|^{q-n}\,dV_K,
\end{align}
and hence
 \eqref{tildeVKZ-def}, \eqref{tildeVKZ-VK} and \eqref{dual-curvature-cone-vol} yield that
	$\widetilde{C}_{n,K}=nV_K$.

Let $q>0$  and $K\in\mathcal{K}_o^n$. For $p\in\R$, 
it follows from \eqref{Cq-nuo-zero} that
it is possible to extend the definition of Lutwak, Yang, Zhang \cite{LYZ18} of the $L_p$ $q$th dual  curvature measure on $S^{n-1}$ as
\begin{equation}
\label{Lpqth-dual-definition}
d\widetilde{C}_{p,q,K}= h_K^{-p}\,d\widetilde{C}_{q,K},
\end{equation}
and hence  for any measurable $\omega\subset S^{n-1}$,  \eqref{dual-curvature-phi-cone-on-boundary} and \eqref{dual-curvature-phi-cone-radial} imply that if $p<1$, then
\begin{align}
\label{lp-dual-curvature-omega-boundary}
\widetilde{C}_{p,q,K}(\omega)=&\int_{\nu_K^{-1}(\omega)\cap\partial' K} \|x\|^{q-n}\langle x,\nu_K(x)\rangle^{1-p}\,d\mathcal{H}^{n-1}(x)\\
\label{dual-curvaturep-pq-radial}
=&\int_{\alpha^*_K(\omega)}h_K\left(\nu_K(\varrho_K(v)\cdot v)\right)^{-p} \cdot\varrho_K(v)^{q}\,d\mathcal{H}^{n-1}(v),
\end{align}
whose value might be infinite. We deduce from \eqref{SKMonge}, \eqref{DhK-length-rhoK} and \eqref{lp-dual-curvature-omega-boundary} that if  $q>0$, $p<1$ and $K\in\mathcal{K}_o^n$ has $C^2_+$ boundary, then
the  Monge-Amp\`ere equation on $S^{n-1}$ corresponding to the $L_p$ $q$th dual Minkowski problem  is
\begin{equation}
\label{Lp-dualMink-Monge-Ampere}
\left(\|\nabla u\|^2+u^2\right)^{\frac{q-n}2}\cdot u^{1-p}\det\left(\nabla^2 u+u\,I\right)=f
\end{equation}
where $u=h_K|_{S^{n-1}}$.
In addition, for a non-negative function $f\in L^1(S^{n-1})$, $q>0$ and $p<1$, a $K\in\mathcal{K}_o^n$ is a solution of the $L_p$ dual Minkowski problem \eqref{Lp-dualMink-Monge-Ampere} in Alexandrov's sense (in the sense of measure) if
\begin{equation}
\label{Lp-dualMink-Monge-Ampere-measure}
d\widetilde{C}_{p,q,K}=f\,d\mathcal{H}^{n-1}.
\end{equation}

As before, let $K\in\mathcal{K}_0^n$. It follows from the definition of $\widetilde{C}_{p,q,K}$ that
$$
\widetilde{C}_{p,q,\lambda\,K}=\lambda^{q-p}\widetilde{C}_{p,q,K} \mbox{ \ for $\lambda>0$}.
$$
As special cases, if $q>0$ and $p<1$, then
\begin{align*}
\widetilde{C}_{0,q,K}=&\widetilde{C}_{q,K},\\
\widetilde{C}_{p,n,K}=&h_K^{1-p}\,dS_K=S_{p,K}\mbox{ \ is Lutwak's so-called $L_p$ surface area.}
\end{align*}
We deduce from \eqref{dual-curvature-phi-cone-surface} and \eqref{dual-curvature-phi-cone-vol} that if $h_K$ is differentiable at each point of a measurable $\omega\subset S^{n-1}$ and $\varphi:\omega\to[0,\infty)$ is measurable, then
\begin{align}
\label{lp-dual-curvature-phi-cone-surface}
\int_{\omega}\varphi\,d\widetilde{C}_{p,q,K}=&
\int_{\omega}\varphi \cdot \|Dh_K\|^{q-n}h_K^{1-p}\,dS_K\\
\label{lp-dual-curvature-phi-cone-vol}
=& n\int_{\omega}\varphi \cdot\|Dh_K\|^{q-n}h_K^{-p}\,dV_K.
\end{align}

We recall that $u=h_K|_{S^{n-1}}$, being Lipschitz, is differentiable $\mathcal{H}^{n-1}$ a.e. on $S^{n-1}$. In addition, if $u$ is differentiable at a $v\in S^{n-1}$, then (cf. \eqref{DhK-u-x} and \eqref{DhK-length-rhoK})
$\langle x,v\rangle=h_K(v)=u(v)$ and
\begin{align*}
x=&Dh_K(v)=\nabla u(v)+u(v)\cdot v,\\
\|x\|=&\varrho_K\circ \alpha^*_K(v)=\sqrt{\|\nabla u(v)\|^2+\|u(v)\|^2}
\end{align*}
where $\nabla u$ is the spherical gradient of $u$. Thus we deduce from  \eqref{dual-curvature-phi-cone-surface} and \eqref{dual-curvature-phi-cone-vol} that if $u$ is differentiable at each point of a measurable $\omega\subset S^{n-1}$, then
	\begin{align}
		\label{dual-curvature-cone-vol1}
		\widetilde{C}_{q,K}(\omega)=&n\int_{\omega}\left(\|\nabla u\|^2+\|u\|^2\right)^{\frac{q-n}2}\cdot u\,dS_K\\
		\label{dual-curvature-cone-vol0}
		=&\int_{\omega}(\varrho_K\circ \alpha^*_K)^{q-n}\,dV_K.
	\end{align}
	Here $u$ is differentiable at each point of a measurable $\omega\subset S^{n-1}$ if and only if $\nu_K^{-1}(\omega)\subset\partial K$ contains no segment.
According to \eqref{dual-curvature-cone-vol1}, 
the  Monge-Amp\`ere equation corresponding to the $q$th dual Minkowski problem for $q>0$ is
\begin{equation}
\label{dualMink-Monge-Ampere}
\left(\|\nabla u\|^2+u^2\right)^{\frac{q-n}2}\cdot u\det\left(\nabla^2 u+u\,I\right)=f,
\end{equation}
and if $p<1$, then the  Monge-Amp\`ere equation corresponding to the $L_p$ $q$th dual Minkowski problem is
\begin{equation}
\label{Lp-dualMink-Monge-Ampere}
\left(\|\nabla u\|^2+u^2\right)^{\frac{q-n}2}\cdot u^{1-p}\det\left(\nabla^2 u+u\,I\right)=f.
\end{equation}

For compact, convex sets $K_1,K_2\subset \R^n$, their Hausdorff distance is just the $L^\infty$ norm of the restrictions of the support functions to $S^{n-1}$; namely, it is  $\min\{|h_{K_1}(x)-h_{K_2}(x)|:x\in S^{n-1}\}$.  
In this paper, convergence of compact, convex sets in $\R^n$ is always meant convergence in terms of the Hausdorff distance. 
We note according the Blaschke selection theorem, any bounded sequence $K_m\in \mathcal{K}^n_o$ of convex bodies contains a subsequence that tends to a compact convex set $K\subset\R^n$ with $O\in K$.
Finally, we sometimes use the following notation: if $X\subset \R^n$, then 
$$
{\rm pos}_+X=\{t x:\,x\in X\mbox{ and }t>0\}.
$$
	
\section{Regularity if the dual Minkowski curvature is bounded and bounded away from zero}
\label{secRegularity}
	
Our technical Lemma~\ref{Gauss-map-bijective} is based on Alexandrov's work and Caffarelli's classical papers \cite{Caf90a,Caf90b,Caf93} (see 
Figalli \cite{Fig17} for a more readable account). 
Let us recall some fundamental properties of Monge-Amp\`ere equations based on Figalli \cite{Fig17} and
 the survey Trudinger, Wang~\cite{TrW08}.
Given a convex function $v$ defined in an open convex set $\Omega\subset\R^{2}$, $D v$ and $D^2 v$ denote its gradient and its Hessian, respectively. The subgradient $\partial v(x)$ of $v$ at $x\in\Omega$ is defined as
$$
\partial v(x) =\{z\in\R^{2} : v(y)\geq v(x)+\langle z,y-x\rangle \text{ for each $y\in\Omega$}\},
$$
which is a compact convex set.
If $\omega\subset\Omega$ is a Borel set, then its associated Monge-Amp\`ere measure is 
\begin{equation}
\label{Monge-Ampere-measure}
\mu_v(\omega)=\mathcal{H}^2\Big(\bigcup_{x\in\omega}\partial v (x)\Big).
\end{equation}
We observe that if $v$ is $C^2$, then
$$
\mu_v(\omega)=\int_\omega \det( D^2 v)\,d\mathcal{H}^{2}.
$$
The following regularity properties of Monge-Amp\`ere equations have been proved in a series of papers by Alexandrov and  Caffarelli, and are thoroughly presented in a simplified manner by Figalli \cite{Fig17}.

\begin{theo}[Caffarelli]
\label{Caffarelli}
For $\lambda>1$, a bounded open convex set $\Omega\subset\R^2$,  and the measurable function $f$ on $\Omega$ with $1/\lambda\leq f\leq \lambda$, let the convex function $v$ on $\Omega$ satisfy
\begin{equation}
\label{Caffarelli-eq}
\det v=f
\end{equation}
in the sense of measure; namely, $d\mu_v=f\,d\mathcal{H}^2$ for the Monge-Amp\`ere measure $\mu_v$.
\begin{description}
		
\item{(i)} $v$ is strictly convex, and $v$ is locally $C^{1,\alpha}$ for an $\alpha\in(0,1)$ depending  on $\lambda$.
			
\item{(ii)} If $f$ is locally $C^{0,\beta}$ for a $\beta\in(0,1)$, then $v$ is locally $C^{2,\beta}$.

			
\end{description}
\end{theo}

The main result of this section is Lemma~\ref{Gauss-map-bijective}, which follows from the properties discussed in Section~\ref{sec-basics}, Caffarelli's Lemma~\ref{Caffarelli}, and from the following observation: For $q>0$, $p<1$, $n=3$, let $K\in\mathcal{K}^3_o$ 
satisfy \eqref{Lp-dualMink-Monge-Ampere}, and let $u=h_K|_{S^{2}}$.
For a fixed $e\in S^{2}$ with $h_K(e)>0$, the function $v(y)=h_K(y+e)$ of $y\in e^\bot$ satisfies that 
$v(y)=(1+\|y\|^2) u(w)$ for $w=(1+\|y\|^2)^{-1}(y+e)\in S^{2}$,  and that the Monge-Amp\`ere equation
\begin{equation}
\label{MongeAmpereRn}
\det( D^2 v(y) )=v(y)^{p-1}\left(\|Dv(y)\|^2+(\langle Dv(y),y\rangle-v(y))^2\right)^{\frac{3-q}2}\cdot g(y) 
\end{equation}
holds for $y\in e^\bot$ in a small neighborhood of the origin $O$ where 
$$
g(y)=\left(1+\|y\|^2\right)^{-\frac{3+p}2} f\left(\frac{e+y}{\sqrt{1+\|y\|^2}}\right).
$$

\begin{lemma}
\label{Gauss-map-bijective}
Let $q>0$, $p<1$, $\lambda>1$ and $K\in\mathcal{K}^3_o$ such that
 $d\widetilde{C}_{p,q,K}=fd\mathcal{H}^{2}$ for a measurable  $f:\,S^{2}\to [\frac1{\lambda},\lambda]$.
\begin{description}
\item[(i)] $h_K$ is $C^{1}$  on the set $\{h_K>0\}$; and $\partial K\backslash\Gamma_K$ is $C^1$ and contains no segment for 
$$
\Gamma_K=\{x\in\partial K:\,\exists w\in\nu_K(x) \mbox{ with } h_K(w)=0\}.
$$
			
\item[(ii)]  If $Z\subset\partial K$ is measurable, and $\omega=\nu_K(Z)$, then
\begin{align}
\label{Gauss-map-bijective-eq}
\widetilde{C}_{p,q,K}(\omega)=&n\int_{Z}\langle x,\nu_K(x)\rangle^{-p}\|x\|^{q-n}\,d\widetilde{V}_K(x)\\
\label{Gauss-map-bijective-SK-eq}
=&n\int_{Z\cap\partial' K}\langle x,\nu_K(x)\rangle^{1-p}\|x\|^{q-n}\,d\mathcal{H}^{n-1}(x).
\end{align}
			
\item[(iii)] $\varrho_K\circ\alpha^*_K(v)=\|Dh_K(v)\|>0$ for both $\mathcal{H}^{n-1}$ a.e. and  $\widetilde{C}_{p,q,K}$ a.e. $v\in S^{n-1}$, and we have
\begin{align}
\label{CpqK-from-cone-volume}
d\widetilde{C}_{p,q,K}=&nh_K^{-p}(\varrho_K\circ\alpha^*_K)^{q-n}\,dV_K=nh_K^{-p}\|Dh_K\|^{q-n}\,dV_K\\
\label{CpqK-from-SK}
=& h_K^{1-p}\|Dh_K\|^{q-n}\,dS_K\\
\label{cone-volume-from-CqK}
dV_K=&\frac{h_K^{p}}n(\varrho_K\circ\alpha^*_K)^{n-q}\,d\widetilde{C}_{p,q,K}=\frac{h_K^{p}}n\cdot \|Dh_K\|^{n-q}\,d\widetilde{C}_{p,q,K}.
\end{align}

\item[(iv)] If $O\in\partial K$, then $\mathcal{H}^{n-1}(\nu_K(o))=0$.

\item[(v)] If, in addition, $f$ is positive and $C^{0,\alpha}$ for an $\alpha\in(0,1)$, then $h_K$ is locally $C^{2,\alpha}$ on $\{h_K>0\}$, and $\partial K\backslash\Gamma_K$ is locally $C^2_+$.

\end{description}
		
\end{lemma}
\noindent{\bf Remark. } Even if $f$ is positive and $C^{0,\alpha}$ for an $\alpha\in(0,1)$, it is possible that $\Gamma_K\subset \partial K$ contains a circular disk centered at $O$ (cf. Example~4.2 in Bianchi, B\"or\"oczky, Colesanti \cite{BBC20}).\\

\section{Basic estimate for Theorem~\ref{n3-ppos--C0-estimate}}
\label{sec-n3-ppos--C0-estimate-basic}

In this section, let $p\in [0,1)$, $q>2+p$ and $\lambda>1$.
 We consider a convex body $K\in\mathcal{K}^3_o$ that satisfies
\begin{equation}
\label{n=3-density-bounded}
d\widetilde{C}_{p,q,K}=f\,d\mathcal{H}^2 \mbox{ \ on $S^2$ with } \ \lambda^{-1}\leq f\leq \lambda
\end{equation} 
for the measurable density function $f$ on $S^2$. 
It follows from  \eqref{CpqK-from-cone-volume}, \eqref{CpqK-from-SK} and \eqref{cone-volume-from-CqK} in Lemma~\ref{Gauss-map-bijective} that
we have
\begin{equation}
\label{n=3-f-density}
f\,d\mathcal{H}^2=d\widetilde{C}_{p,q,K}=3h_K^{-p}\|Dh_K\|^{q-3}\,dV_K=h_K^{1-p}\|Dh_K\|^{q-3}\,dS_K
\end{equation}
and 
\begin{equation}
\label{dual equation n=3}
dV_K=\frac13\cdot h_K^{p}\|Dh_K\|^{3-q}\,d\widetilde{C}_{p,q,K}=
\frac13\cdot h_K^{p}\|Dh_K\|^{3-q}f\,d\mathcal{H}^2.
\end{equation}
According to Example~4.2 in Bianchi, B\"or\"oczky, Colesanti \cite{BBC20}, it might happen that $O\in\partial K$.

To keep track of the size of convex body $K\in\mathcal{K}^3_o$, we use the unique ellipsoid $E\subset K$ of maximal volume, which is the so-called John ellipsoid of $K$ (see 
Ball \cite{Bal97} or Gruber, Schuster \cite{GrS05} for the uniqueness and other basic properties of the John ellipsoid). We will write $X=(X_1,X_2,X_3)$ to denote the center of the John ellipsoid of $K$, and hence 
\begin{equation}
\label{John n=2}
E\subset K\subset X+3(E-X).
\end{equation}
Let $r_1(K),r_2(K),r_3(K)>0$ be the half axes of the John ellipsoid $E$ of $K$ where we assume that
\begin{equation}
\label{r123}
r_1(K)\leq r_2(K)\leq r_3(K).
\end{equation}
Let $e_1,e_2,e_3$ be the fixed orthonormal basis of $\R^3$. Since our problem is rotation invariant, we may also assume that
$e_1,e_2,e_3$ form the principal directions of $E$; namely,
\begin{equation}
\label{riei}
r_i(K)e_i\in\partial E, \mbox{ \ }i=1,2,3.
\end{equation}
In the rest of the paper, for any $K\in\mathcal{K}^3_o$ satisfying \eqref{n=3-density-bounded}, we use the notation as in \eqref{John n=2}, \eqref{r123} and \eqref{riei}. In addition, we use the following notation concerning various quantities:
We write $a \gtrsim b$ (resp. $a \lesssim b$) if there exists a constant $C>1$ depending only on  $\lambda$, $p$ and $q$, such that $a \geq Cb$ (resp.$a \leq Cb$), and $a \approx b$ means that $C^{-1}b \leq a \leq Cb$.

Lemma~\ref{q3 main} provides our ``basic estimate" that will be used frequently throughout the paper.

\begin{lemma} \label{q3 main}
For $p\in [0,1)$, $q>2+p$ and $\lambda>1$,  if $K\in\mathcal{K}^3_o$ 
satisfies \eqref{n=3-density-bounded}, and $r_i=r_i(K)$, $i=1,2,3$, then
\begin{equation}\label{q3 main equality}
    r_1r_2r_3 \approx r_3^{3-q+p}.
\end{equation}
\end{lemma}
\begin{proof}
We recall that $X=(X_1,X_2,X_3)$ is the center of the John ellipsoid $E\subset K$ in \eqref{John n=2}. Without loss of generality, we can assume $X_3 \geq 0$. Denote 
$$
G:=\left\{x \in \partial K: \langle x, e_3\rangle \geq \langle X , e_3\rangle+\frac{1}{4}r_3\right\}
$$
and 
$$
F:=\{\nu_K(x):x\in G\}.
$$
\begin{figure}
    \centering
    \includegraphics[width=1\linewidth]{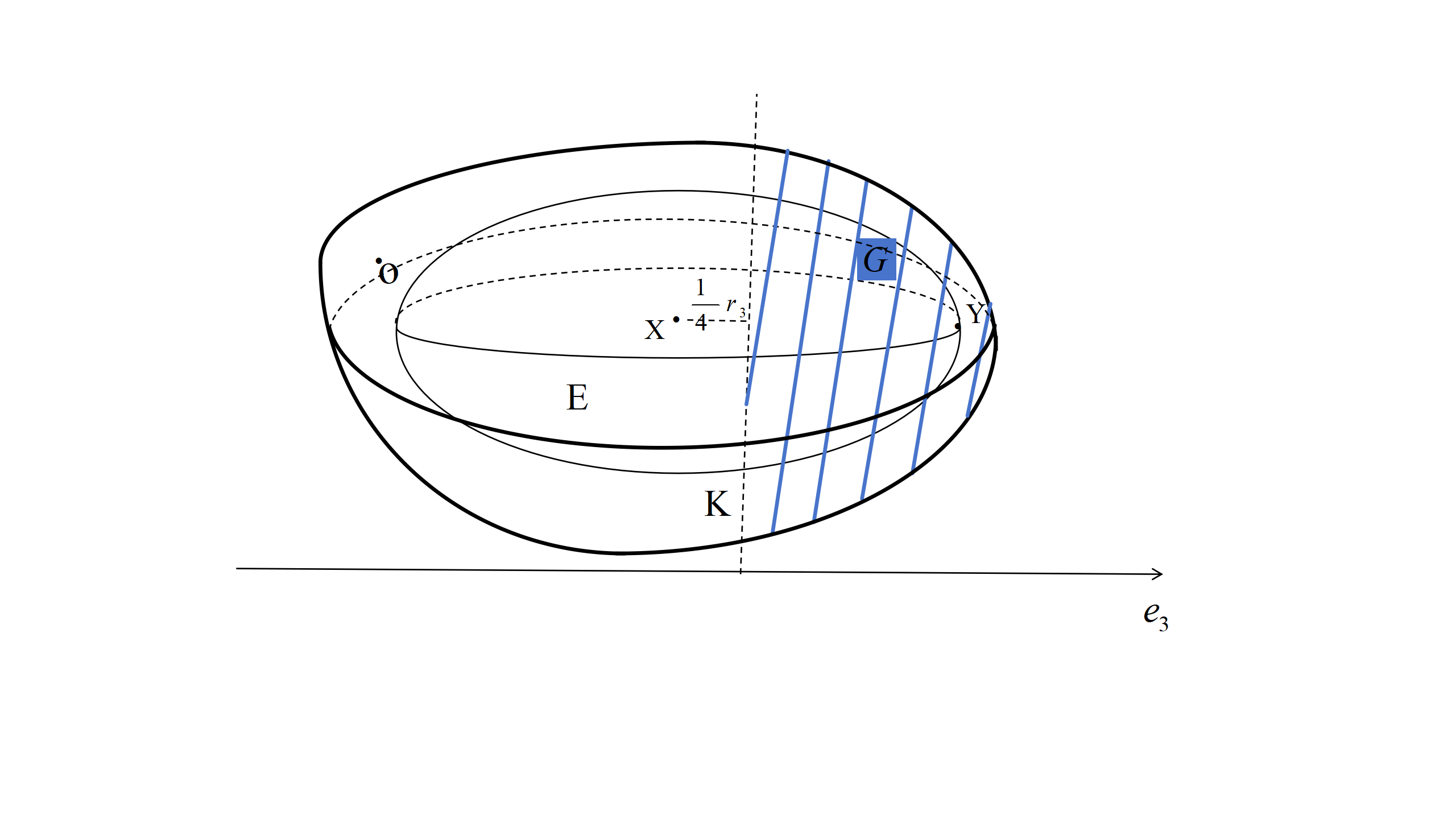}
    \caption{lemma \ref{q3 main}}
\end{figure}
Our next task is to verify  that $\mathcal{H}^2(F) \approx 1$ (see \eqref{set bound}).  
We note that the point $Y=(Y_1,Y_2,Y_3)=(X_1,X_2,X_3+r_3)\in \partial E$ belongs to $K$. For any $y=(y_1,y_2,y_3)\in \partial K \backslash G$ and $\nu_y=(\nu_{y,1},\nu_{y,2},\nu_{y,3})\in\nu_K(y)$,
by convexity we have
 $\langle Y-y,\nu_y\rangle \leq 0.$
 By the choice of $Y$ and $y$, we have 
   $Y_3-y_3 >\frac{3}{4}r_3>0.$ It follows from \eqref{John n=2} that 
   $y_1-Y_1\leq 12r_1$ and $y_2-Y_2\leq 12r_2.$
   Hence, 
    if $\nu_{y,1} >0$ and $\nu_{y,2} >0$, then
\begin{equation*}
\begin{aligned}
    \nu_{y,3} \leq \frac{y_1-Y_1}{Y_3-y_3}\cdot \nu_{y,1} +\frac{y_2-Y_2}{Y_3-y_3}\cdot \nu_{y,2}
        \leq 16\cdot \nu_{y,1}+16\cdot \nu_{y,2}.
\end{aligned}
\end{equation*}

Let $\mathbb{S}_{100}:=\{\nu=(\nu_1,\nu_2,\nu_3) \in  S^2: \nu_3\leq 100\cdot \nu_1+100\cdot \nu_2\}$, and $ S^2_{+}:=\{\nu=(\nu_1,\nu_2,\nu_3) \in  S^2: \nu_1 > 0,\nu_2 > 0\}$, then we have
\begin{equation}
\label{set lower bound}
    F \supseteq  S^2_{+} \backslash \mathbb{S}_{100},
\end{equation}
so,
\begin{equation}\label{set bound}
    \mathcal{H}^2(S^2)\geq \mathcal{H}^2(F) \geq \mathcal{H}^2\left(S^2_{+} \backslash \mathbb{S}_{100}\right)>0.
\end{equation}
On the one hand, for $\mathcal{H}^2$ a.e. $\nu \in F$, we have 
\begin{equation}\label{rho bound}
    \frac{1}{4}r_3 \leq ||Dh_K||(\nu) \leq 12 r_3.
\end{equation}
For $\nu=(\nu_1,\nu_2,\nu_3) \in  S^2_{+} \backslash \mathbb{S}_{100}$ and $z=(z_1,z_2,z_3)\in \nu^{-1}_K(\nu)\cap G$, we have $\nu_3 > 100\cdot \nu_1+100\cdot \nu_2>0.$ 
Since $|\nu|=1,$ it follows that $\nu_3 \geq \frac{1}{2}$. 
Note by \eqref{John n=2} we have that $z_1\geq -12r_1, z_2\geq -12r_2.$ By the definition of $G$,
we have $z_3\geq \frac{1}{4}r_3.$
Thus, 
\begin{equation}\label{h lower bound}
    \begin{aligned}
        h_K(\nu) &= z_1\nu_1+z_2\nu_2+z_3\nu_3\\
        &\geq  -12r_1 \nu_1- 12r_2 \nu_2+\frac{1}{4}r_3 \nu_3 \\
        &\geq  (-12r_1 \nu_1-12r_2 \nu_2+\frac{1}{8}r_3 \nu_3)+\frac{1}{8}r_3 \nu_3 \\
        &\geq \frac{1}{16}r_3.
    \end{aligned}
\end{equation}
Then, $p\geq 0$,  the estimate $\lambda^{-1}\leq f\leq \lambda$ in \eqref{n=3-density-bounded} and \eqref{dual equation n=3} yield that  
\begin{equation}\label{q3 main ineq1}
\begin{aligned}
V_K(F)&=\frac{1}{3}\int_{F}h_K^p(\nu) ||Dh_K||^{3-q}(\nu)f(\nu)d\nu\\
& \geq \frac{1}{3}\int_{ S^2_{+} \backslash \mathbb{S}_{100}}h_K^p(\nu) ||Dh_K||^{3-q}(\nu)f(\nu)d\nu\\
& \gtrsim r_3^{3-q+p},
\end{aligned}
\end{equation}
where the second inequality follows from \eqref{set lower bound} and the last inequality follows from \eqref{rho bound} and \eqref{h lower bound}.
Moreover, using $h_K(\nu)\leq ||Dh_K(\nu)||$ in \eqref{DhK-rhoK}, we deduce that
\begin{equation}\label{q3 main ineq2}
\begin{aligned}
V_K(F)&=\frac13\int_{F}h_K^p(\nu) ||Dh_K||^{3-q}(\nu)f(\nu)d\nu\\
& \leq \frac13\int_{F}||Dh_K||^{3-q+p}(\nu)f(\nu)d\nu\\
& \lesssim r_3^{3-q+p},
\end{aligned}
\end{equation}
where the last inequality follows from \eqref{set bound} and \eqref{rho bound}.
Let $\mathcal{C}_G:=\{ty: t\geq 0, y\in G\}$, then 
$|\mathcal{C}_G\cap K|\leq V_K(F)\leq |K|$ (cf. \eqref{tildeVKZ-def2}). For $Z=X+\frac12\,r_3e_3\in E$, if $x\in\frac{15}{16}Z+\frac1{16}E\subset E$, then $\langle x,e_3\rangle>\langle X,e_3\rangle+\frac14\,r_3$, and hence $\frac{15}{16}Z+\frac1{16}E\subset \mathcal{C}_G$; therefore, 
\begin{equation}\label{main q3 ineq3}
\begin{aligned}
r_1r_2r_3\approx |E| \lesssim V_K(F) \leq  |K|\leq (2\cdot 3)^3 r_1r_2r_3.
\end{aligned}
\end{equation}
Therefore, by combining the equalities \eqref{q3 main ineq1}, \eqref{q3 main ineq2} and \eqref{main q3 ineq3}, we obtain $ r_1 r_2 r_3 \approx r_3^{3-q+p} $.
\end{proof}

\begin{coro} 
\label{q3 estimate} 
For $p\in [0,1)$, $q>2+p$ and $\lambda>1$,  if $K\in\mathcal{K}^3_o$ 
satisfies \eqref{n=3-density-bounded},  then
$$
r_3(K)^{q-p} \gtrsim \frac{r_3(K)}{r_1(K)}, \mbox{ \ \ }
r_3(K) \gtrsim 1\mbox{ and }r_1(K) \lesssim 1.
$$
\end{coro}
\begin{proof} Let $r_i=r_i(K)$ for $i=1,2,3$.
    Since $r_1 \leq r_2 \leq r_3$, \eqref{q3 main equality} yields that $r_3^{q-p} \gtrsim \frac{r_3}{r_1}$ and $r_1^{q-p} \lesssim 1$, which in turn imply that $r_3 \gtrsim 1$ and $r_1 \lesssim 1$.
\end{proof}


Given Corollary~\ref{q3 estimate}, which is a consequence of Lemma~\ref{q3 main}, Theorem~\ref{n3-ppos--C0-estimate} will follow from the following statement:

\begin{prop}
\label{good-inequality-prop}
For $p\in [0,1)$, $q>2+p$ and $\lambda>1$,  if $K\in\mathcal{K}^3_o$ 
satisfies \eqref{n=3-density-bounded},  then
\begin{equation}
\label{good inequality}
    r_1(K)\gtrsim r_3(K).
\end{equation}
\end{prop}

\section{Proof of Proposition~\ref{good-inequality-prop}, and in turn of Theorem~\ref{n3-ppos--C0-estimate}}
\label{sec-n3-ppos--C0-estimate}

In this section, we fix $p\in [0,1)$, $q>2+p$ and $\lambda>1$, and keep using the notation introduced in Section~\ref{sec-n3-ppos--C0-estimate-basic}. Indirectly, we suppose that Proposition~\ref{good-inequality-prop} does not hold, and hence  one of Case I or Case II occurs: \\

\noindent {\bf Case I}: 
$$
r_1 \leq r_2\ll r_3
$$
may hold for $r_i=r_i(K)$, $i=1,2,3$; that is, for arbitrary large $M>1$, there exists a solution $K\in\mathcal{K}^3_o$ to the equation \eqref{n=3-density-bounded}  satisfying that 
\begin{equation}\label{q3 case I}
  \frac{r_3(K)}{r_2(K)}>M.
\end{equation}

\noindent {\bf Case II}: 
$$
r_1 \ll r_2 \approx r_3
$$
may hold for $r_i=r_i(K)$, $i=1,2,3$; that is, for a constant $C_1>1$ depending only on $\lambda$, $p$, $q$  such that for any $M>1$ large, there exists a solution $K$ to the equation \eqref{n=3-density-bounded} satisfying that
\begin{equation}\label{q3 case II}
 \frac{r_2(K)}{r_1(K)}>M\ \text{and }\ \frac{r_3(K)}{r_2(K)}\leq C_1.
\end{equation}

\subsection{Lemmas needed for the proof of Proposition~\ref{good-inequality-prop} under the assumption as in Case I ($r_1 \leq r_2\ll r_3$)}

For the $K\in\mathcal{K}^3_o$ 
satisfying \eqref{n=3-density-bounded} and \eqref{q3 case I}, let $I=[a,b]e_3$ be the orthogonal projection of $K$ on the $e_3$-axis where $a\leq 0\leq b$. Without loss of generality, we may assume 
$$
|a| \geq b\geq 0.
$$
Now Case I can be subdivided into the following two subcases:

\medskip
\noindent {\it Subcase (i)}: for some constant $c_0\in(0,\frac{1}{3})$, for any $M>0$, there exists a convex body 
$K\in\mathcal{K}^3_o$ satisfying \eqref{n=3-density-bounded},  \eqref{q3 case I} and  
\begin{equation}\label{q3 case I subcase i condition}
    \mbox{dist}(O,\partial I)\geq c_0r_3.
\end{equation}
{\it Subcase (ii)}: for any $M>1$, there exists a convex body $K\in\mathcal{K}^3_o$ satisfying \eqref{n=3-density-bounded}, \eqref{q3 case I} and 
\begin{equation}\label{q3 case I subcase i condition 2}
  \mbox{dist}(O,\partial I)<  \frac{r_3}{M}.
\end{equation}

\begin{lemma}\label{q3 case I subcase i lemma}
Subcase (i) in Case I does not occur.
\end{lemma}
\begin{proof}
In this argument, the implied constants (independent of $M$) in $\gtrsim$ and $\approx$ depend on the $c_0$ in \eqref{q3 case I subcase i condition} besides $\lambda,p,q$.
Indirectly, we suppose that subcase (i) in Case I occurs, take $\sigma_0=\frac{1}{3}c_0$, and define
\begin{eqnarray*}
G_{1}:=\{x=(x_1,x_2,x_3)\in \partial K: \mbox{dist} (x_3e_3, \partial I)\geq \sigma_0r_3\},
\end{eqnarray*}
\begin{eqnarray*}
F_{1}:=\{\nu\in  S^2: \nu\in\nu_K(x)\ \text{for some}\ x\in G_{1}\},
\end{eqnarray*}
and see Figure \ref{q3 case I subcase i lemma figure}.

\begin{figure}
    \centering
    \includegraphics[width=1\linewidth]{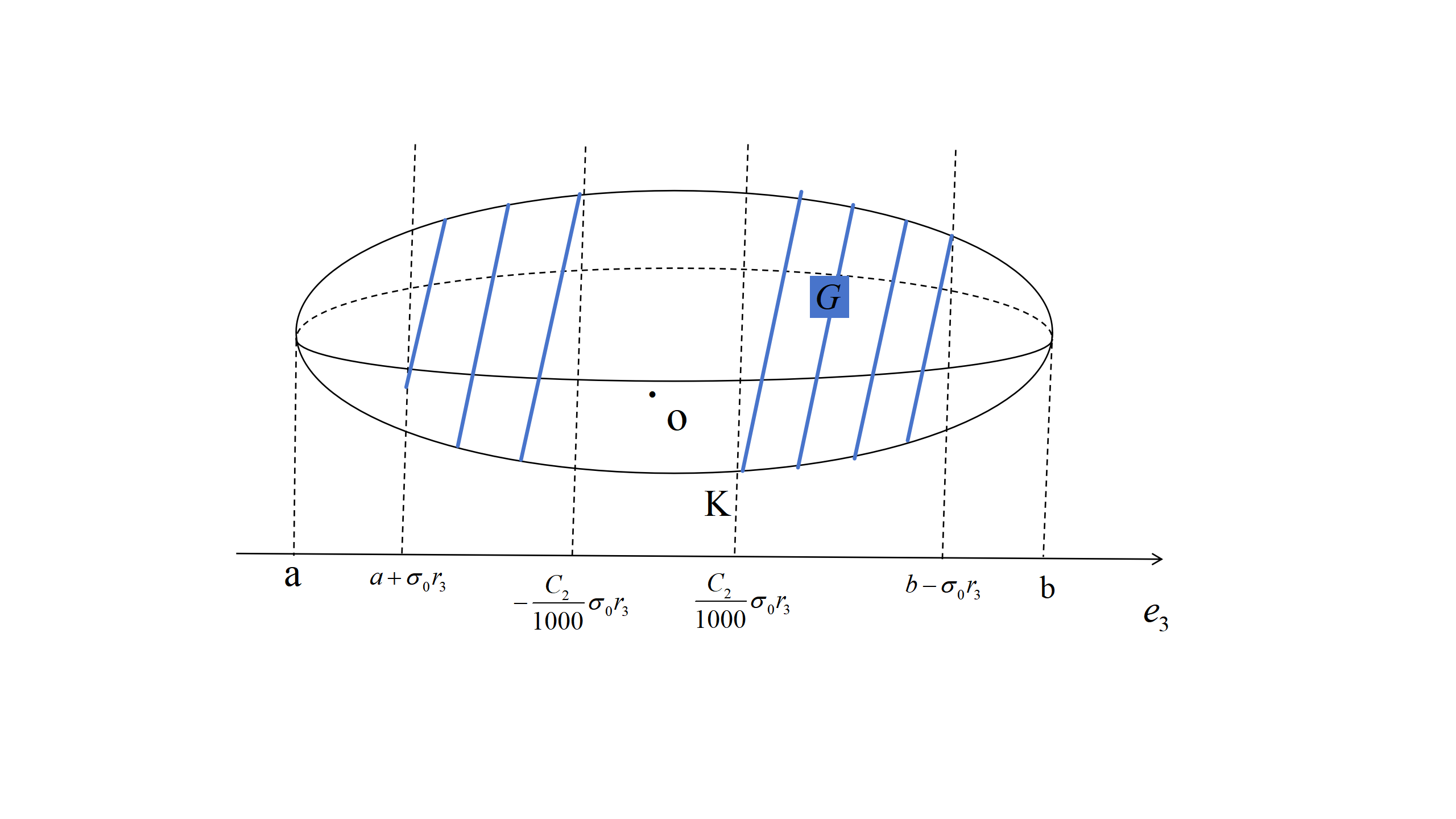}
    \caption{lemma \ref{q3 case I subcase i lemma}}
    \label{q3 case I subcase i lemma figure}
\end{figure}
On the one hand, for $\nu=(\nu_1,\nu_2,\nu_3) \in F_1$,  we have $|\nu_3|\leq \frac{96}{\sigma_0 M}$ by the convexity of $K$. Indeed, given any $x=(x_1,x_2,x_3)\in G_1$, by the definition of $G_1$ we can find $y=(y_1,y_2,x_3+\frac{1}{2}\sigma_0 r_3)\in K$ and $z=(z_1,z_2,x_3-\frac{1}{2}\sigma_0 r_3)\in K$. Note that \eqref{John n=2} yields that $|x_i|,|y_i|,|z_i|\leq 2\cdot 3\cdot r_i\leq \frac{12}{M}r_3$ for $i=1,2$. For any $\nu=(\nu_1,\nu_2,\nu_3)\in \nu_K(x)$, by convexity we have $\nu\cdot(y-x)\leq 0$. Hence,
\begin{equation}\label{normal 3 estimate}
    \frac{1}{2}\sigma_0\nu_3 r_3 \leq |(x_1 - y_1)\nu_1| + |(x_2 - y_2)\nu_2| \leq \frac{48}{M}r_3.
\end{equation}
Therefore, $\nu_3\leq \frac{96}{\sigma_0 M}$. Similarly, changing $y$ to $z$ in the above argument,  we have $\nu_3\geq -\frac{96}{\sigma_0 M}$. Thus, $|\nu_3|\leq \frac{96}{\sigma_0 M}$, and we deduce that
\begin{equation}\label{q3 case I subcase i ineq1}
\mathcal{H}^2(F_{1}) \lesssim \frac{96}{\sigma_0 M}\to 0
\end{equation}
as $M\to+\infty$.

On the other hand, the argument leading to \cite[estimate (3.7)]{CFL22} shows that $\bigcup\{{\rm conv}\{o,x\}:\,x\in G_1\}$ contains a translated copy of $\tilde{c}_0E$ where $\tilde{c}_0\in(0,1)$ depends  on $c_0,\lambda,p,q$, and hence (cf. \eqref{tildeVKZ-def2})
\begin{equation}
\label{q3 case I subcase i volumn estimate}
V_K(F_{1})=\left|\bigcup\{{\rm conv}\{o,x\}:\,x\in G_1\}\right|\geq C_2r_1r_2r_3
\end{equation}
for some constant $C_2\in(0,1)$ independent of $M$.
Let 
$$
G_{2}:=\left\{x=(x_1,x_2,x_3)\in \partial K: |x_3|\leq \frac{C_2}{1000}\,\sigma_0r_3\right\},
$$
and 
$$
F_2:=\{\nu\in  S^2: \nu\in\nu_K(x)\ \text{for some}\ x\in G_2\},
$$
then $V_K(F_{2})\leq \frac{1}{2}C_2r_1r_2r_3$.

For $G=G_{1} \backslash G_{2}$, and
\begin{eqnarray*}
F:=\{\nu\in  S^2: \nu\in\nu_K(x)\ \text{for some}\ x\in G\},
\end{eqnarray*}
we have 
$$
\mathcal{H}^2(F)\leq \mathcal{H}^2(F_{1})\leq \frac{96}{\sigma_0 M}
$$
and (cf. \eqref{q3 case I subcase i volumn estimate})
$$
V_K(F) \geq \frac{C_2}{2}\,r_1r_2r_3.
$$
Moreover, for $\mathcal{H}^2$ a.e. $\nu \in F$, we have 
\begin{equation}\label{rho bound 2}
    12r_3 \geq \|Dh_K(\nu)\| \geq \frac{C_2}{1000}\,\sigma_0 r_3.
\end{equation}
Therefore, \eqref{n=3-f-density} yields
\begin{equation}
\label{q3 case I subcase i ineq2}
\begin{aligned}
\mathcal{H}^2(F_{1}) &\geq \mathcal{H}^2(F) \approx \widetilde{C}_{p,q,K}(F)=3\int_{F}h_K^{-p}||Dh_K||^{q-3}dV_K(\nu) \\
&\gtrsim r^{q-p-3}_3V_K(F)
\gtrsim r_3^{q-p-3}r_1r_2r_3 
\gtrsim 1
\end{aligned}
\end{equation}
where the above inequalities follow from \eqref{rho bound 2}, \eqref{q3 case I subcase i volumn estimate} and \eqref{q3 main equality}, respectively.
When M is sufficiently large, \eqref{q3 case I subcase i ineq1} contradicts \eqref{q3 case I subcase i ineq2}, so subcase (i) of Case I does not occur.
\end{proof}

Now we start to rule out subcase (ii) of Case I. Recall that in subcase (ii), \eqref{q3 case I} and
\eqref{q3 case I subcase i condition 2}  hold, namely  $\frac{r_3(K)}{r_2(K)}>M$ and
 $\mbox{dist}(O,\partial I)< \frac{r_3}{M}$. We subdivide subcase (ii) of Case I further, as the quotient $\frac{r_2(K)}{r_1(K)}$ is bounded in Lemma~\ref{q3 case I subcase ii.1 lemma},  and the quotient $\frac{r_2(K)}{r_1(K)}$ can be arbitrary large in Lemma~\ref{q3 case I subcaseii.2 lemma}.

\begin{lemma}\label{q3 case I subcase ii.1 lemma}
In subcase (ii) of Case I, the case $r_1 \approx r_2\ll r_3$  does not occur; namely, the assumption that there exists some $C_3>1$ such that for any $M>1$ large, the equation \eqref{n=3-density-bounded} has a solution $K\in\mathcal{K}^3_o$  satisfying that
\begin{equation}
\label{CaseIsubcaseii-boundedr2r1}
\frac{r_3(K)}{r_2(K)}>M, \mbox{ and }\mbox{dist}(O,\partial I)<\frac{r_3}{M}, \mbox{ and }
\frac{r_2(K)}{r_1(K)}<C_3
\end{equation}
leads to a contradiction.
\end{lemma}

\begin{proof}
Assume that \eqref{CaseIsubcaseii-boundedr2r1} holds in Case I subcase (ii), and let $r_i=r_i(K)$, $i=1,2,3$. 
In this argument, the implied constants in $\gtrsim$ and $\approx$ depend on the $C_3$ in \eqref{CaseIsubcaseii-boundedr2r1} besides $\lambda,p,q$, and hence these constants are still independent of $M$.
We consider
\begin{align*}
G_{1}:=&\{x=(x_1, x_2, x_3)\in \partial K: x_3 \leq b-\delta r_3\},\\
F_{1}:=&\{\nu\in  S^2: \nu\in\nu_K(x)\ \text{for some}\ x\in G_1\},
\end{align*}
where $\delta=\frac{384r_2}{r_3}$ converges to zero as M converges to $+\infty$. See Figure \ref{q3 case I subcase ii.1 lemma figure}.
\begin{figure}
    \centering
    \includegraphics[width=1\linewidth]{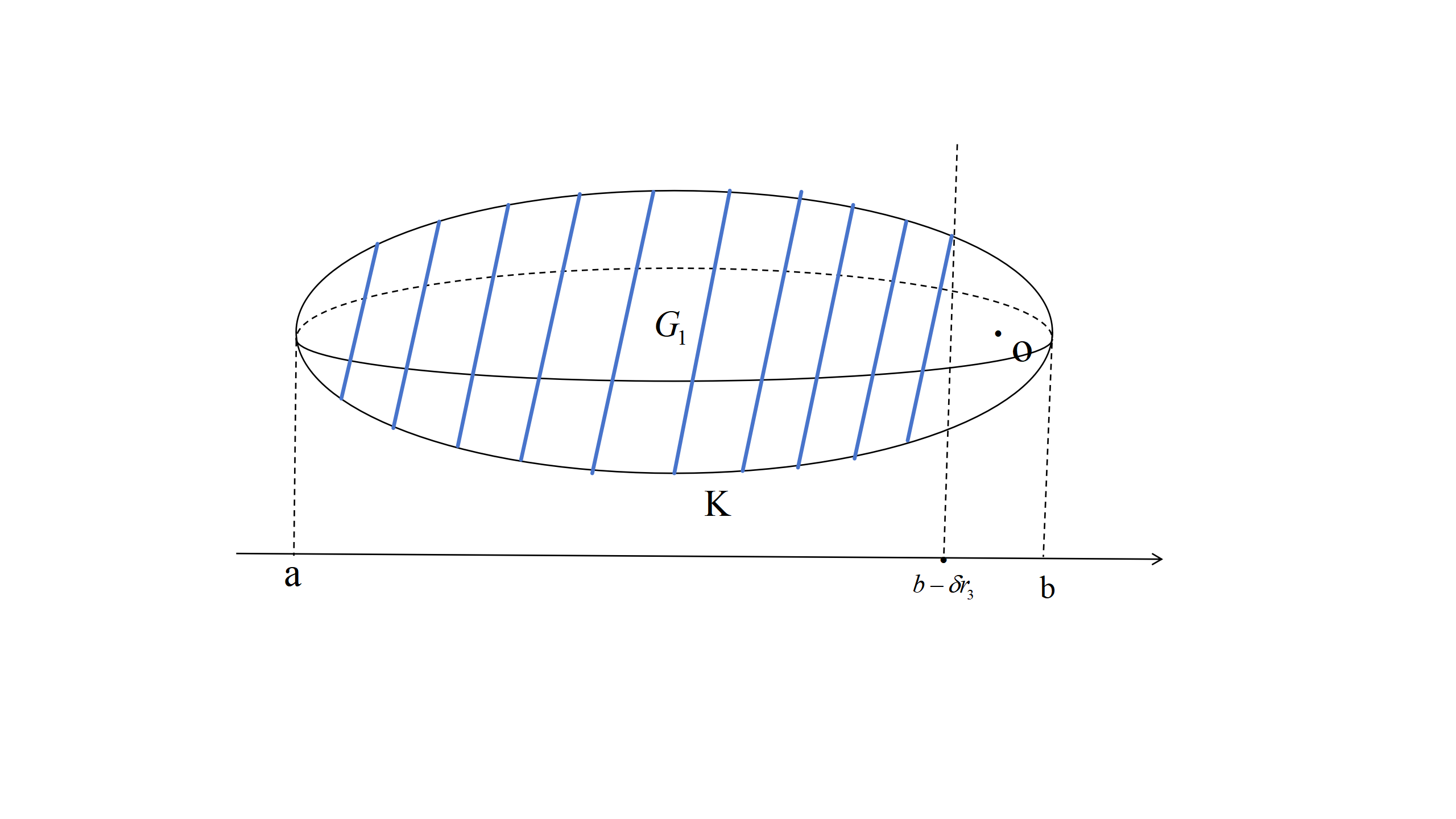}
    \caption{lemma \ref{q3 case I subcase ii.1 lemma}}\label{q3 case I subcase ii.1 lemma figure}
    \label{fig:enter-label}
\end{figure}
On the one hand, by the convexity of $K$, we have that 
$$
F_{1}\subset\mathbb{S}_{\frac{1}{4}}:=\left\{\nu\in  S^2: \langle\nu, e_3\rangle \leq {\frac{1}{4}}\right\}.
$$
Indeed, given any $x=(x_1, x_2, x_3)\in G_1,$ we can find $y=(y_1, y_2, x_3+\frac{1}{2}\delta r_3)\in K.$ Note that $|x_i|, |y_i|\leq 12r_2$ for $i=1, 2.$
For any $\nu=(\nu_1, \nu_2, \nu_3)\in \nu_K(x),$ we have $\langle \nu, y-x\rangle\leq 0$ by the convexity of $K$.
Hence, 
\begin{equation}\label{angle upper bound}
    \frac{1}{2}\delta \nu_3r_3\leq |(x_1-y_1)\nu_1|+|(x_2-y_2)\nu_2|\leq 48 r_2.
\end{equation}
We deduce that 
\begin{equation*}
    \nu_3\leq \frac{96r_2}{\delta r_3} = \frac{1}{4}
\end{equation*}
by $\delta =\frac{384r_2}{r_3},$ which in turn implies $F_{1} \subset\mathbb{S}_{\frac{1}{4}}.$
It follows that 
$$
G:=\left\{x\in \partial K: \nu_K(x) \in  S^2 \backslash \mathbb{S}_{\frac{1}{4}}\right\} \subset \partial K \backslash G_{1}.
$$

For $\mathcal{H}^2$ a.e. $\nu\in  S^2 \backslash \mathbb{S}_{\frac{1}{4}},$ we have that $Dh_K(\nu)\in \partial K \backslash G_1.$ Thus,  
 \eqref{q3 case I subcase i condition 2}, \eqref{q3 case I} and the definiton of $G_1,$ 
yield that 
\begin{equation}\label{small rho bound}
    ||Dh_K(\nu)|| \lesssim \frac{r_3}{M}.
\end{equation}
Since $\delta =\frac{384r_2}{r_3}$, using the assumption $r_1\approx r_2 \ll r_3$ (cf. \eqref{CaseIsubcaseii-boundedr2r1}), we have 
\begin{equation}\label{area upper bound}
    \mathcal{H}^{2}(G) \leq\mathcal{H}^{2}(\partial K \backslash G_{1}) \lesssim \delta r_2r_3+r_1r_2 \lesssim r^2_1.
\end{equation}

For $p\in[0,1)$ and $q > 2+p$ , we deduce from \eqref{n=3-f-density} that
\begin{equation}\label{5.4 upper bound}
\begin{aligned}
\widetilde{C}_{p,q,K}( S^2 \backslash \mathbb{S}_{\frac{1}{4}})&=\int_{ S^2 \backslash \mathbb{S}_{\frac{1}{4}}}h^{1-p}_K(\nu)||Dh_K||^{q-3}(\nu)dS_K(\nu)\\
&=\int_{ S^2 \backslash \mathbb{S}_{\frac{1}{4}}}\left(\frac{h_K}{||Dh_K||}\right)^{1-p}||Dh_K||^{q-p-2}dS_K\\
&\lesssim \left(\frac{r_3}{M}\right)^{q-p-2}\mathcal{H}^2(G)\\
&\lesssim  \left(\frac{1}{M}\right)^{q-p-2}r_3^{q-p-2}r^2_1\\
& \lesssim \left(\frac{1}{M}\right)^{q-p-2},
\end{aligned}
\end{equation}
where the first inequality follows from \eqref{small rho bound}, the second inequality follows from \eqref{area upper bound}, and the last inequality follows from the basic estimate \eqref{q3 main equality} and the assumption  $r_1 \approx r_2\ll r_3$.

Since
\begin{equation}\label{q3 case I subcase ii.1 lemma lower bound}
C_{p,q,K}( S^2 \backslash \mathbb{S}_{\frac{1}{4}})
=\int_{ S^2 \backslash \mathbb{S}_{\frac{1}{4}}}f(\nu)d\nu
\geq\frac{1}{\lambda}\mathcal{H}^{2}\left( S^2 \backslash \mathbb{S}_{\frac{1}{4}}\right)
\end{equation}
where the last expression is a constant depending only on $\lambda$,  for sufficiently large $M$, we find that \eqref{5.4 upper bound} contradicts \eqref{q3 case I subcase ii.1 lemma lower bound}. In turn, we conclude Lemma~\ref{q3 case I subcase ii.1 lemma} that assumes the condition $r_1 \approx r_2\ll r_3$ in subcase (ii) of Case I. 
\end{proof}

\begin{lemma}\label{q3 case I subcaseii.2 lemma}
In subcase (ii) of Case I, $r_1 \ll r_2\ll r_3$ does not occur; namely, the assumption that  for any $M>1$ large, there exists a solution $K\in\mathcal{K}^3_o$ to the equation \eqref{n=3-density-bounded} satisfying that
\begin{equation}
\label{CaseIsubcaseii-unboundedr2r1}
\frac{r_3(K)}{r_2(K)}>M, \mbox{ and }\mbox{dist}(O,\partial I)<\frac{r_3}{M}, \mbox{ and }
\frac{r_2(K)}{r_1(K)}>M
\end{equation} 
leads to a contradiction.
\end{lemma}

\begin{proof}
We recall that $I=[a, b]e_3$ is the orthogonal projection of $K$ on the $e_3$-axis.
As $|a|\geq b\geq 0$ has been assumed, we have $\mbox{dist}(O,\partial I)=b\leq \frac{r_3}{M}$.
Define
\begin{eqnarray*}
G_1:=\{x=(x_1, x_2, x_3)\in \partial K: x_3 \leq b-\delta r_3\},
\end{eqnarray*}
\begin{eqnarray*}
F_1:=\{\nu\in  S^2: \nu\in\nu_K(x)\ \text{for some}\ x\in G_1\},
\end{eqnarray*}
where $\delta=\frac{384r_2}{r_3}$ converges to zero as M converges to $+\infty$.
\begin{figure}
    \centering
    \includegraphics[width=1\linewidth]{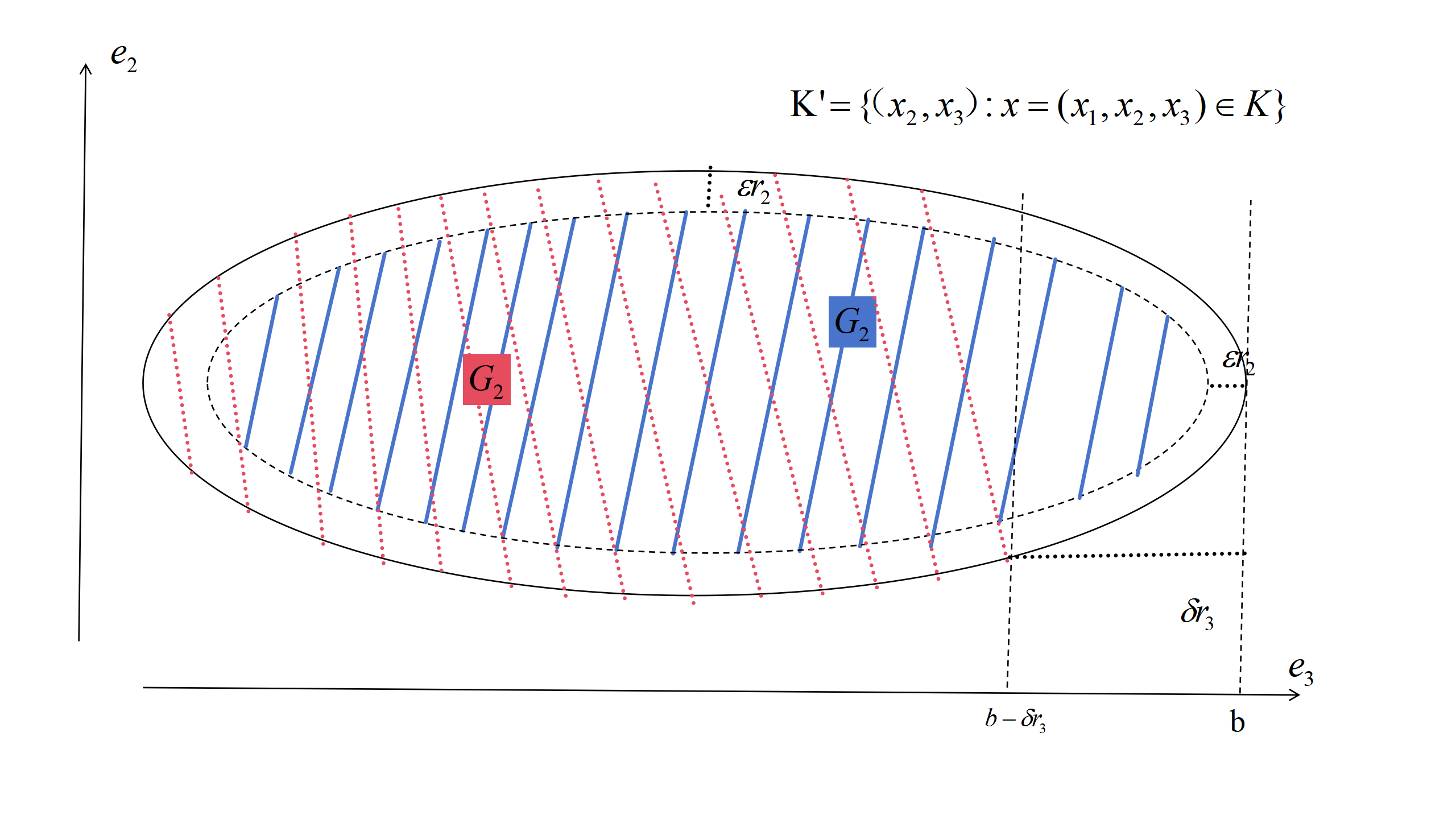}
    \caption{lemma \ref{q3 case I subcaseii.2 lemma}}
\end{figure}
On the one hand, similar to the proof of \eqref{angle upper bound}, for any $\nu=(\nu_1, \nu_2, \nu_3)\in F_1 $, we have $$\nu_3\leq \frac{96 r_2}{\delta r_3} = \frac{1}{4},$$
which implies 
$$
F_1\subset\mathbb{S}_{\frac{1}{4}}:=\left\{\nu\in  S^2: \langle \nu, e_3 \rangle\leq \frac{1}{4}\right\}.$$

Denote by $K^\prime$ the projection of $K$ on the $e_2e_3$ plane. 
Let  $\varepsilon=\frac{96 r_1}{r_2}$, that can be arbitrary small by \eqref{CaseIsubcaseii-unboundedr2r1}, and let
\begin{align*}
G_2:=&\{x=(x_1, x_2, x_3)\in \partial K:~ \mbox{dist} (x^\prime, \partial K^\prime)\geq \varepsilon r_2,~ x^\prime=(x_2, x_3)\},\\
F_2:=&\{\nu\in  S^2: \nu\in\nu_K(x)\ \text{for some}\ x\in G_2\}.
\end{align*}
For any $x=(x_1, x_2, x_3)\in G_2$, by the choice of $G_2$, there exist points $y=(y_1, x_2+\varepsilon r_2, x_3)\in K$
and $\tilde{y}=(\tilde{y}_1, x_2-\varepsilon r_2, x_3)\in K$.
 Note that $|x_1|, |y_1|,|\tilde{y}_1|\leq 12r_1$. For any $\nu=(\nu_1, \nu_2, \nu_3)\in \nu_K(x)$, we have that
$\langle \nu, y-x\rangle\leq 0$ and $\langle \nu, \tilde{y}-x\rangle\leq 0$, and hence  a straightforward computation yields that
\begin{equation}
\label{normal vector estimate 1}
    \pm\varepsilon r_2\nu_2\leq \max\{|\nu_1(y_1-x_1)|,|\nu_1(\tilde{y}_1-x_1)|\}\leq 24 r_1.
\end{equation}
It follows that
\begin{equation}
\label{normalvectorestimate-nu2r1r2}
    |\nu_2|\leq \frac{24 r_1}{\varepsilon r_2}.
\end{equation}
Similarly, replacing $y$ and $\tilde{y}$ by $z=(y_1, x_2, x_3+\varepsilon r_2)\in K$ and $\tilde{z}=(y_1, x_2, x_3-\varepsilon r_2)\in K$, we  deduce
\begin{equation}
    |\nu_3|\leq \frac{24r_1}{\varepsilon r_2}.
\end{equation}
Therefore,  we obtain $F_2 \subset\mathbb{S}_{\frac{1}{4}}$ as $\varepsilon  = \frac{96 r_1}{r_2}$.
In particular, we have
$$
F_1\bigcup F_2\subset\mathbb{S}_{\frac{1}{4}}:=\left\{\nu\in  S^2: \langle \nu, e_3\rangle \leq {\frac{1}{4}}\right\}
$$
as $\delta=\frac{384 r_2}{r_3}$ and $\varepsilon=\frac{96 r_1}{r_2}$, and in turn, we deduce that
$$
G:=\left\{x\in \partial' K: \nu_K(x) \in  S^2 \backslash \mathbb{S}_{\frac{1}{4}} \right\}\subset \partial K \backslash (G_{1}\cup G_2),
$$
Since $\mbox{dist}(O,\partial I)<\frac{r_3}{M}$ by the condition \eqref{CaseIsubcaseii-unboundedr2r1}, the definiton of $G_1$ and \eqref{CaseIsubcaseii-unboundedr2r1} yield that
$$
\text{diam}(\partial K \backslash G_{1})\lesssim \max\{\delta r_3, r_1, r_2\}\lesssim r_2\leq \frac{r_3}{M}.$$
 Combining this estimate with \eqref{CaseIsubcaseii-unboundedr2r1} implies that
 $|y|\lesssim \frac{r_3}{M}$ for any $y\in G.$

Therefore, for $\mathcal{H}^2$ a.e. $\nu \in  S^2 \backslash \mathbb{S}_{\frac{1}{4}}$, we have
\begin{equation}\label{small rho bound 2}
    ||Dh_K(\nu)|| \lesssim \frac{r_3}{M}.
\end{equation}
Since $\varepsilon =\frac{96r_1}{r_2}$ and $\delta=\frac{384r_2}{r_3}$,
\begin{equation}\label{area uppper bound 2}
    \mathcal{H}^{2}(\partial K \backslash (G_1 \cup G_2)) \lesssim  r_1 (\varepsilon r_2+\delta r_3)\approx r_1r_2.
\end{equation}

For $p\in[0,1)$ and $q > 2+p$, we deduce from \eqref{n=3-f-density} that
\begin{equation}\label{5.5 upper bound}
\begin{aligned}
\widetilde{C}_{p,q,K}\left( S^2 \backslash \mathbb{S}_{\frac{1}{4}}\right)&=\int_{ S^2 \backslash \mathbb{S}_{\frac{1}{4}}}h^{1-p}_K(\nu)||Dh_K||^{q-3}(\nu)dS_K(\nu)\\
&=\int_{ S^2 \backslash \mathbb{S}_{\frac{1}{4}}}\left(\frac{h_K}{||Dh_K||}\right)^{1-p}||Dh_K||^{q-p-2}dS_K(\nu)\\
&\leq \left(\frac{r_3}{M}\right)^{q-p-2}\mathcal{H}^2(G)\\
&\lesssim  \left(\frac{1}{M}\right)^{q-p-2}r_3^{q-p-2}r_1r_2\\
& \lesssim \left(\frac{1}{M}\right)^{q-p-2},
\end{aligned}
\end{equation}
where the first inequality follows from \eqref{small rho bound 2}, the second inequality follows from \eqref{area uppper bound 2}, and the last inequality follows from the basic estimate \eqref{q3 main equality}.
Since
\begin{equation*}\label{q3 case I subcaseii.2 lemma lower bound}
\widetilde{C}_{p,q,K}\left( S^2 \backslash \mathbb{S}_{\frac{1}{4}}\right)=\int_{ S^2 \backslash \mathbb{S}_{\frac{1}{4}}}f(\nu)d\nu
\geq\frac{1}{\lambda}\mathcal{H}^{2}\left( S^2 \backslash \mathbb{S}_{\frac{1}{4}}\right)>0,
\end{equation*}
the above estimate contradicts \eqref{5.5 upper bound} when $M$ is sufficiently large, completing the argument for Lemma~\ref{q3 case I subcaseii.2 lemma} (ruling out subcase (ii) of Case I).
\end{proof}

After considering Case I, we continue the proof of Theorem~\ref{n3-ppos--C0-estimate} under the assumptions in Case II.\\

\subsection{Lemmas needed for the proof of Proposition~\ref{good-inequality-prop} under the assumption as in Case II ($r_1 \ll r_2 \approx r_3$)}
\label{subsecCaseII}

Given $p\in[0,1)$, $q >2+p$ and $\lambda>1$, Case II says that there exists a constant $C_1>1$ depending on $p,q,\lambda$, such that for any large $M>1$, one finds a solution $K$ to the equation \eqref{dual equation n=3} with   density function $f$ of $\widetilde{C}_{p,q,K}$ satisfying
\begin{equation}\label{case2 hold}
 1/\lambda<f<\lambda,\ \text{and that}\ \frac{r_2(K)}{r_1(K)}>M\ \text{and}\ \frac{r_3(K)}{r_2(K)}\leq C_1.
\end{equation}
We keep $\lambda,p,q,C_1$ fixed for the whole Section~\ref{subsecCaseII} about Case II. For any object $Y\subset\R^3$, we write $Y'$ to denote the orthogonal projection of $Y$ into the $e_2e_3$ plane.

Within Case II, one of the following two subcases can occur.

\medskip
 \noindent {\it Subcase (i)}: There exists $c_4>0$ depending on $p,q,\lambda$, such that for any large $M>1$, one finds a solution $K$ to the equation \eqref{dual equation n=3} satisfying \eqref{case2 hold} and
\begin{equation}
\label{subcasei-case2}
\mbox{dist} (O, \partial K^\prime)> c_4 r_3.
\end{equation}
\noindent {\it Subcase (ii)}: For any large $M>1$, one finds a solution $K$ to the equation \eqref{dual equation n=3} satisfying \eqref{case2 hold} and
\begin{eqnarray}
\label{CaseIIsubcase(ii)}
\mbox{dist} (O, \partial K^\prime)\leq \frac{r_3}{M}.
\end{eqnarray}

\begin{lemma}\label{q3 case II subcase i lemma}
Subcase (i) in Case II does not occur.
\end{lemma}
\begin{proof} In the proof of Lemma~\ref{q3 case II subcase i lemma}, the implied constants in $\lesssim$ and $\approx$ depend also on the $C_1$ and $c_4$ introduced in \eqref{case2 hold} and \eqref{subcasei-case2} besides $\lambda,p,q$. For any object $Y\subset\R^3$, we write $Y'$ to denote its projection in the coordinate plane ${\rm lin}\{e_2,e_3\}$.

We suppose that subcase (i) in Case II occurs, and hence \eqref{case2 hold} and \eqref{subcasei-case2} hold, and seek a contradiction. We consider (see Figure \ref{q3 case II subcase i lemma figure})
\begin{eqnarray}\label {G1def}
G_{1}:=\left\{x\in \partial' K:~ \mbox{dist} (x^\prime, \partial K^\prime)\geq \frac{c_4 r_3}{4}\right\},
\end{eqnarray}
\begin{eqnarray*}
F_{1}:=\{\nu\in  S^2: \nu\in\nu_K(x)\ \text{for some}\ x\in G_{1}\}.
\end{eqnarray*}
\begin{figure}
    \centering
    \includegraphics[width=1\linewidth]{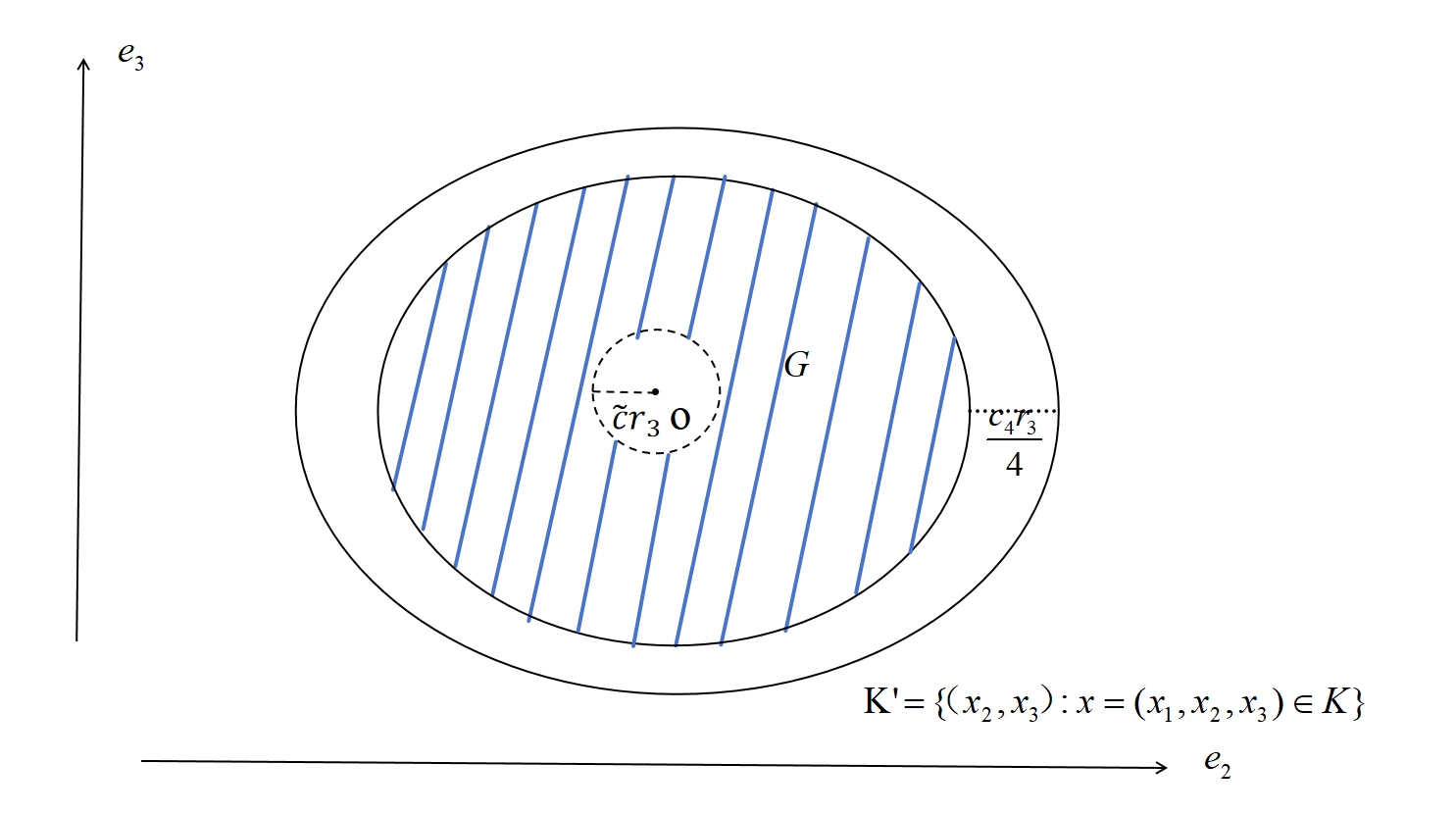}
    \caption{lemma \ref{q3 case II subcase i lemma}}
    \label{q3 case II subcase i lemma figure}
\end{figure}
Similarly to the proof of \eqref{normal vector estimate 1}, for any $\nu=(\nu_1,\nu_2,\nu_3)\in F_1$, we can deduce
$|\nu_2|\leq \frac{96}{Mc_4}$ and $|\nu_3|\leq \frac{96}{Mc_4}$.
Hence, the density function $f$ of $\widetilde{C}_{p,q,K}$ satisfies
\begin{equation}
\label{q3 case II subcase i ineq1}
\widetilde{C}_{p,q,K}(F_1)=\int_{F_{1}}f\lesssim \mathcal{H}^2(F_1)\lesssim \frac{1}{M^{2}}
\end{equation} 
provided $M$ is sufficiently large.

To obtain a lower bound for $\widetilde{C}_{p,q,K}(F_1)$, first we estimate $V_K(F_{1})$. We observe that the origin centered ellipsoid $\widetilde{E}=E-X$ with half axes $r_1,r_2,r_3$ satisfies
\begin{equation}
\label{widetildeE-John}
X+\widetilde{E}\subset K\subset X+3\widetilde{E}\subset 6\widetilde{E}
\end{equation} 
by \eqref{John n=2}, and
$$
\widetilde{G}_1=\left\{x\in \partial' K:~ x^\prime\in \frac{c_4}{2}\,\widetilde{E}'\right\}\subset \left\{x\in \partial' K:~ \|x^\prime\|\leq \frac{c_4 r_3}{2}\right\}\subset G_1
$$
by \eqref{subcasei-case2}. We claim that $c_5= \frac{4\pi}3\cdot\frac{c_4^2}{252}$ satisfies that
\begin{equation}
\label{q3-case2-subscasei-F1C3}
V_K(F_{1})\geq \widetilde{V}_K(\widetilde{G}_1)\geq c_5r_1r_2r_3.
\end{equation}
There exists a $z_0\in\partial K$ such that $z'_0\in\partial K'$ and $O$ lies on the segment connecting $z'_0$ and $X'$, and hence \eqref{subcasei-case2} and imply that $O=(1-t)z'_0+tX'$ where $\frac{c_4}3\leq t\leq 1$, and in turn we deduce 
via \eqref{widetildeE-John}  that $z=(1-t)z_0+tX$ satisfies that $z=sr_1e_1$ for $s\in[-6,6]$ and 
$z+\frac{c_4}3\,\widetilde{E}\subset K$. We may assume that $s\geq 0$ by possibly changing $e_1$ into $-e_1$, and hence
for $\tilde{s}=s+\frac{c_4}6$, we have
\begin{equation}
\label{q3-case2-subscasei-tildes}
\tilde{s}r_1e_1+\frac{c_4}6\,\widetilde{E}\subset K\mbox{ \ and \ }\tilde{s}\geq \frac{c_4}6.
\end{equation}
Since $7r_1e_1+\frac{c_4}6\,\widetilde{E}\subset {\rm pos}_+\widetilde{G}_1$  by \eqref{widetildeE-John} and the definition of $\widetilde{G}_1$ , we have 
$$
\tilde{s}r_1e_1+\frac{c_4\tilde{s}}{42}\,\widetilde{E}\subset K\cap {\rm pos}_+\widetilde{G}_1=\bigcup\{{\rm conv}\{o,x\}:\,x\in \widetilde{G}_1\}
$$
by \eqref{q3-case2-subscasei-tildes},
proving the claim \eqref{q3-case2-subscasei-F1C3} by \eqref{tildeVKZ-def2}.


Next, we define the constant $\tilde{c}\in(0,1)$ by the formula
$$
 6\pi\tilde{c}^2C_1=\frac{c_5}{2}
$$
using the $C_1$ in \eqref{case2 hold},
and let
\begin{align*}
G_{2}:=&\left\{x=(x_1,x_2,x_3)\in \partial' K: \sqrt{|x_2|^2+|x_3|^2}\leq \tilde{c} \cdot r_3\right\},\\
\nonumber
F_{2}:=&\{\nu\in  S^2: \nu\in\nu_K(x)\ \text{for some}\ x\in G_{2}\}.
\end{align*}
For 
$$
\Gamma_K=\{x\in\partial K:\,\exists w\in\nu_K(x) \mbox{ with } h_K(w)=0\},
$$
Lemma~\ref{Gauss-map-bijective} (i) yields any segment in $\partial K$ is contained in $\Gamma_K$. Therefore, \eqref{tildeVKZ-VK}, \eqref{tildeVKZ-def2} and the definition of $G_2$ imply that
$$
V_K(F_{2})\leq \widetilde{V}_K(G_2\cup \Gamma_K)=\widetilde{V}_K(G_2)=
\left|\bigcup\{{\rm conv}\{o,x\}:\,x\in G_2\}\right|\leq \left|{\rm conv}\,G_2\right|.
$$
In particular, $\tilde{c}$ is chosen in a way such that the last inequality in \eqref{q3-case2-subscasei-c3} holds in the chain of inequalities 
\begin{equation}
\label{q3-case2-subscasei-c3}
V_K(F_{2})\leq \left|{\rm conv}\,G_2\right|\leq 6r_1\cdot \pi(\tilde{c} r_3)^2\leq 
6\pi (\tilde{c} r_3)^2C_1r_1r_2r_3
\leq  \frac{c_5}{2}\,r_1r_2r_3
\end{equation}
where the inequalities in \eqref{q3-case2-subscasei-c3} hold by the definition of $G_2$ and $r_3\leq C_1r_2$.
For $G=G_{1} \backslash G_{2}$, and
\begin{eqnarray*}
F:=\{\nu\in  S^2: \nu\in\nu_K(x)\ \text{for some}\ x\in G\},
\end{eqnarray*}
we have $\mathcal{H}^2(F)\leq\mathcal{H}^2(F_{1})\lesssim \frac{1}{M^2}$ by \eqref{q3 case II subcase i ineq1} and $V_K(F) \geq \frac{c_5}{2}\,r_1r_2r_3$ by \eqref{q3-case2-subscasei-F1C3} and \eqref{q3-case2-subscasei-c3}. Moreover, for $\mathcal{H}^2$ a.e. $\nu \in F$, we have $12r_3\geq ||Dh_K(\nu)|| \geq \tilde{c}\, r_3$, and hence  \eqref{n=3-f-density} yields that
\begin{equation}\label{q3 case II subcase i ineq2}
\begin{aligned}
\widetilde{C}_{p,q,K}(F_1)\geq   \widetilde{C}_{p,q,K}(F)= & \int_{F}h_K^{-p}||Dh_K||^{q-3}dV_K(\nu) \\
   \gtrsim &r^{q-p-3}_3V_K(F)\\
    \gtrsim & r_3^{q-p-3}r_1r_2r_3 \\
    \gtrsim &1,
\end{aligned}
\end{equation}
where the last inequality follows from \eqref{q3 main equality}.
When M is sufficiently large, \eqref{q3 case II subcase i ineq1} contradicts \eqref{q3 case II subcase i ineq2}, so subcase (i) of Case II does not occur.
\end{proof}

Let us turn to subcase (ii) of Case II, where we have $r_1\ll r_2 \approx r_3$; namely, for any large $M>1$,
 there exists  a solution $K$ to the equation \eqref{dual equation n=3}  such that  \eqref{case2 hold} holds, and
\begin{equation}
\label{subcaseii-CaseII-repeat}
\mbox{dist} (O, \partial K^\prime)\leq  \frac{r_3}{M}  \ \text{and}\ 
\frac{r_2}{r_1}>M \ \text{while}\ \frac{r_3}{r_2}\leq C_1
\end{equation}
hold for $r_i=r_i(K)$, $i=1,2,3$,
where $C_1>1$ is independent of $M$.
We recall that the center of the John ellipsoid $E$ in \eqref{John n=2} is $X=(X_1,X_2,X_3)$, and hence 
$$
h_{\widetilde{K}}(\nu)=h_K(\nu)-\langle\nu, X\rangle
$$
for $\nu\in S^2$
is the support function of convex body $\widetilde{K}=K-X$. 
The main argument for subcase (ii) of Case II in Lemma~\ref{q3 case II subcase ii.3 lemma} is prepared by Lemma~\ref{q3 case II subcase ii.1 lemma}
 and Lemma~\ref{q3 case II subcase ii.2 lemma}.


\begin{lemma}
\label{q3 case II subcase ii.1 lemma}
If subcase (ii) of Case II holds and the $M$ in \eqref{subcaseii-CaseII-repeat} is large enough, then
there exists a unit vector $\nu=(\nu_1,\nu_2,\nu_3)\in \nu_K(G_+)$ such that 
\begin{equation}
\label{nu123-estimate}
\nu_1\geq \frac{1}{2}, \quad |\nu_2| \lesssim \frac{r_1}{r_3}, \quad \text{and} \quad |\nu_3| \lesssim \frac{r_1}{r_3}, 
\end{equation}
and $\nu$ satisfies
\begin{equation}
\label{nu-estimate}
h_K(\nu) \leq r_3^{2-3(q-p)},
\end{equation}
where $G_+:=\left\{(x_1,x_2,x_3)\in \partial K: x_1>X_1, \frac{1}{8}\,r_i<|x_i-X_i|<\frac{1}{2}\,r_i, i=2,3\right\}$,
and the implied constants in $\lesssim$ in \eqref{nu123-estimate} depend also on the $C_1$ in \eqref{subcaseii-CaseII-repeat} besides $\lambda$, $p$, $q$. 
\end{lemma}
\begin{proof}
First, for any $\nu=(\nu_1,\nu_2,\nu_3)\in \nu_K(G_+)$, and hence $\nu\in \nu_K(x)$ for some $x=(x_1, x_2, x_3)\in G_+$,  there exist points $y=(y_1, x_2+\frac{r_3}{2C_1}, x_3)\in K$,
and $\tilde{y}=(\tilde{y}_1, x_2-\frac{r_3}{2C_1}, x_3)\in K$; moreover, $z=(y_1, x_2, x_3+\frac{r_3}{2})\in K$,
and $\tilde{z}=(\tilde{y}_1, x_2, x_3-\frac{r_3}{2})\in K$; therefore, 
similarly to the proof of \eqref{normalvectorestimate-nu2r1r2}, we deduce that 
\begin{equation}\label{F1 area}
|\nu_2|\lesssim \frac{r_1}{r_2} \lesssim \frac{r_1}{r_3},\quad \text{and} \quad |\nu_3| \lesssim \frac{r_1}{r_3}.
\end{equation}
Since $\|\nu\|=1$ and $\frac{r_1}{r_3}\leq \frac{1}{M}$,  it follows that 
$\nu_1\geq \frac{1}{2}$ for large $M$.

Indirectly, we suppose that there does not exist a unit vector $\nu=(\nu_1,\nu_2,\nu_3)\in \nu_K(G_+)$ such that 
\eqref{nu-estimate} holds, and hence 
\begin{equation}\label{5.7 h upper bound}
    h_K(\nu) > r_3^{2-3(q-p)}
\end{equation}
for any $\nu\in \nu_K(G_+)$. 

Combining the above inequalities and $|X_i|\lesssim r_i$ for $i=1,2,3$, we have 
\begin{eqnarray*}
\begin{aligned}
-\nu \cdot X \leq |\nu_1||X_1|+|\nu_2||X_2|+|\nu_3||X_3| \lesssim r_1.
\end{aligned}
\end{eqnarray*}
Since $r_2\approx r_3,$ by \eqref{q3 main equality},  we deduce
 that $r_1\approx r_3^{1-q+p}.$

Then, by  \eqref{5.7 h upper bound}, we have
\begin{eqnarray*}
\begin{aligned}
-\nu \cdot X   &\lesssim r_1\cdot r_3^{3(q-p)-2} h_K(\nu) \approx r_3^{2(q-p)-1}h_K(\nu).
\end{aligned}
\end{eqnarray*}
Thus,
\begin{equation}\label{useful inequality}
\begin{aligned}
&h_{\widetilde{K}}h_K^{-p}||Dh_K||^{q-3}dS_{\widetilde{K}} (\nu)\\
= & (h_K(\nu)-\langle \nu, X\rangle)h_K^{-p}||Dh_K||^{q-3}dS_{K} (\nu)\\
\lesssim & (r_3^{2(q-p)-1}+1)h_K^{1-p}(\nu)||Dh_K||^{q-3}dS_{K} (\nu)\\
\lesssim &r_3^{2(q-p)-1}f(\nu)d\nu.\\
\end{aligned}
\end{equation}
Let $F_+:=\nu_K(G_+)$, and hence  $\mathcal{H}^2(F_+)\lesssim (\frac{r_1}{r_3})^2$ by \eqref{F1 area}.
It follows that the density function $f$ of $\widetilde{C}_{p,q,K}$ satisfies
\begin{equation}\label{q3 case II subcase ii.1 ineq1}
\begin{aligned}
\int_{F_+}f  \leq \lambda\mathcal{H}^2(F_+) \lesssim \left(\frac{r_1}{r_3}\right)^2\approx r_3^{-2(q-p)}.
\end{aligned}
\end{equation}
On the other hand, let us estimate $V_{\widetilde{K}}(F_+)$.
For the centered ellipsoid $\widetilde{E}=E-X$,  \eqref{John n=2} implies that $\widetilde{E}\subset \widetilde{K}\subset 3\widetilde{E}$. It follows from $\widetilde{E}\subset \widetilde{K}\subset 3\widetilde{E}$ and \eqref{tildeVKZ-def2} that
\begin{equation}
\label{q3-case2-subcaseii-VtildeKF1-tildeG1}
V_{\widetilde{K}}(F_+)\geq \widetilde{V}_{\widetilde{K}}(\widetilde{G}_+)
\end{equation}
for $\widetilde{G}_+:=\left\{(x_1,x_2,x_3)\in \partial \widetilde{E}: x_1>0, \frac{1}{8}r_i<|x_i|<\frac{1}{6}r_i, i=2,3\right\}$. Now $\widetilde{E}=\Phi(B^n)$ for the diagonal transform satisfying $\Phi e_i=r_ie_i$, $i=1,2,3$, and hence
\begin{align*}
c_4=&
\widetilde{V}_{B^3}\left(\left\{(x_1,x_2,x_3)\in \partial B^n: x_1>0, \frac{1}{8}<|x_i|<\frac{1}{6}, i=2,3\right\}\right)\\
=&\widetilde{V}_{B^3}\left(\Phi^{-1}\widetilde{G}_+\right)
\end{align*}
is a positive absolute constant, and in turn \eqref{q3-case2-subcaseii-VtildeKF1-tildeG1} and the linear equivariance property \eqref{tildeVK-linear-inv} yields that
\begin{equation}
\begin{aligned}
\label{q3-case2-subcaseii-VtildeKF1}
    V_{\widetilde{K}}(F_+)&\geq \widetilde{V}_{\widetilde{K}}(\widetilde{G}_+)=\det \Phi\cdot \widetilde{V}_{\Phi^{-1} \widetilde{K}}\left(\Phi^{-1}\widetilde{G}_+\right)\\
   &\geq \det \Phi\cdot \widetilde{V}_{B^3}\left(\Phi^{-1}_*\widetilde{G}_+\right)=c_4 r_1r_2r_3.
\end{aligned}
\end{equation}

Next, it follows from \eqref{John n=2} and  the definition of $F_+, G_+$ that for any $\nu \in F_+$, we have
\begin{equation}\label{another inequality}
    h_{K}(\nu)  \leq 12 r_3 \quad \text{and} \quad \frac{1}{2}r_3\leq ||Dh_{K}||(\nu) \leq 12r_3.
\end{equation}
In turn, we deduce that
\begin{equation}\label{q3 case II subcase ii.1 ineq2}
\begin{aligned}
\int_{F_+}f &\gtrsim r_3^{1-2(q-p)}\int_{F_1}h_{\widetilde{K}}h_K^{-p}||Dh_K||^{q-3}dS_{\widetilde{K}}\\
& \gtrsim r_3^{-(q-p)-2}\int_{F_1}h_{\widetilde{K}}(\nu)dS_{\widetilde{K}}(\nu)=n r_3^{-(q-p)-2}V_{\widetilde{K}}(F_+)\\
&\gtrsim r_3^{-(q-p)-2}r_1r_2r_3\\
& \gtrsim r_3^{1-2(q-p)},\\
\end{aligned}
\end{equation}
where the first inequality follows from \eqref{useful inequality}, the second inequality follows from \eqref{another inequality}, the third inequality follows from \eqref{q3-case2-subcaseii-VtildeKF1}, and the last inequality follows from \eqref{q3 main equality}.

We deduce from Corollary~\ref{q3 estimate} that $r_3\gtrsim (\frac{r_3}{r_1})^{\frac1{q-p}}\geq M^{\frac1{q-p}}$. Since $M$ can be arbitrarily large,  \eqref{q3 case II subcase ii.1 ineq1} contradicts \eqref{q3 case II subcase ii.1 ineq2}, which in turn yields Lemma \ref{q3 case II subcase ii.1 lemma}.
\end{proof}

\begin{lemma}\label{q3 case II subcase ii.2 lemma}
If subcase (ii) of Case II holds and the $M$ in \eqref{subcaseii-CaseII-repeat} is large enough, then
there exists a unit vector $\xi=(\xi_1,\xi_2,\xi_3)\in \nu_K(G_-)$ such that 
\begin{equation}
\label{xi123-estimate}
\xi_1\leq -\frac{1}{2}, \quad |\xi_2| \lesssim \frac{r_1}{r_3}, \quad \text{and} \quad |\xi_3| \lesssim \frac{r_1}{r_3}, 
\end{equation}
and $\xi$ satisfies
\begin{equation}
\label{xi-estimate}
h_K(\xi) \leq r_3^{2-3(q-p)},
\end{equation}
where $G_-:=\left\{(x_1,x_2,x_3)\in \partial K: x_1<X_1, \frac{1}{8}\,r_i<|x_i-X_i|<\frac{1}{2}\,r_i, i=2,3\right\}$,
and the implied constants in $\lesssim$ in \eqref{xi123-estimate} depend also on the $C_1$ in \eqref{subcaseii-CaseII-repeat} besides $\lambda$, $p$, $q$. 
\end{lemma}
\begin{proof}
The proof is similar to the proof of  Lemma \ref{q3 case II subcase ii.1 lemma}, while changing $G_+$ to $G_-$.
\end{proof}

Next, we  prove that subcase (ii) in Case II does not occur.

\begin{lemma}\label{q3 case II subcase ii.3 lemma}
For given $p\in[0,1)$, $q>2+p$ and $\lambda>1$, subcase (ii) in Case II does not occur.
\end{lemma}
\begin{proof} 
In the argument for Lemma~\ref{q3 case II subcase ii.3 lemma}, the implied constants in $\lesssim$, $\gtrsim$ and $\approx$ depend also on the $C_1$ in \eqref{subcaseii-CaseII-repeat} besides $\lambda$, $p$, $q$.
We suppose that subcase (ii) of Case II occurs, and seek a contradiction. Then, 
there exists a sequence of $(K_m,f_m)$, $m \in \mathbb{N}$, where the convex body $K_m\in\mathcal{K}^3_o$ satisfies
$d\widetilde{C}_{p,q,K_m}=f_m\,d\mathcal{H}^2$ on $S^2$ for the measurable function
$f_m:S^2\to[\frac1{\lambda},\lambda]$, and there exists $C_1>1$ such that 
\begin{equation}
\label{5.8 assumption}
    r_3(K_m)\leq C_1 r_2(K_m) \text{ and } \lim_{m\to\infty}\frac{r_2(K_m)}{r_1(K_m)} =
\lim_{m\to\infty}\frac{r_3(K_m)}{r_1(K_m)} =\infty;
\end{equation}
moreover, the condition \eqref{CaseIIsubcase(ii)} of subcase (ii) yields that
\begin{equation}\label{5.8 assumption 2}
    \lim_{m\to\infty}\frac{\mbox{dist}(O,\partial K'_m)}{r_3(K_m)} =0
\end{equation}
 where  $K_m^\prime$ is the orthogonal projection of $K_m$ on the $e_2e_3$ plane. Possibly taking a subsequence of the``blown down" convex bodies
\begin{eqnarray*}
 \widetilde{K}_m:=\frac{1}{r_3(K_m)}\,K_m
\end{eqnarray*}
provided by the Blaschke selection theorem, we may assume by \eqref{John n=2} and \eqref{5.8 assumption} that
\begin{equation}
\label{tildeKm-to-Kinfty}
 \lim_{m\to\infty}\widetilde{K}_m= K_\infty\subset \mathbb{R}e_2\oplus \mathbb{R}e_3=\mathbb{R}^2
\end{equation}
in Hausdorff distance where $K_\infty$ is a two-dimensional compact convex set and $O\in \partial K_\infty$ by 
\eqref{5.8 assumption 2}. We write $\partial' K_\infty$ to denote the set of $x\in \partial K_\infty$ where there exists a unique tangent line in $\R^2$ at $x$ to $K_\infty$, and hence $\mathcal{H}^1(\partial K_\infty\backslash \partial' K_\infty)=0$. For $x\in \partial K_\infty$, $\nu_{K_\infty}(x)$ is the set of exterior  normal vectors $\nu\in S^1$ at $x$ to $K_\infty$.
By a rotation, we may assume that
\begin{equation}
\label{n3-Kinfty-Onormal}
-e_3\in \nu_{K_\infty}(O) \mbox{ \ and \ }t_\infty e_3\in {\rm relint}\,K_\infty\mbox{ for some $t_\infty>0$,}
\end{equation}
and hence
$$
K_\infty\subset \{(0,x_2,x_3)\in\R^3:x_3\geq 0\}.
$$
We recall the spherical coordinates for $S^2$:
$$
T^2:~(\varphi,\theta)\in \left(-\frac{\pi}{2}, \frac{\pi}{2}\right) \times (-\pi, \pi) \mapsto (\sin\varphi, \cos\varphi\sin\theta,-\cos \varphi\cos\theta)\in  S^2
$$
that actually parametrize $S^2\backslash\{(\sin\varphi,0,-\cos\varphi):\,\varphi\in[-\frac{\pi}2,\frac{\pi}2]\}$,
and for $S^1=S^2\cap\R^2$, $S^1\backslash \{e_3\}$ can be parametrized based on arc length as
$$
T^1:~\theta\in  (-\pi, \pi) \mapsto (0,\sin\theta,-\cos \theta )\in  S^1
$$
where $\theta$ is the angle of $\nu=(0,\sin\theta,-\cos \theta )$ and $-e_3$ and $T^1(\theta)=T^2(0,\theta)$.

For $\varepsilon\in(0,t_\infty)$, we consider
\begin{align*}
D_\varepsilon:=&\{(0,x_2,x_3)\in\partial K_\infty:\,x_3< \varepsilon\},\\ 
\widetilde{F}_\varepsilon:=&\{\nu\in  S^1: \nu\in\nu_{K_\infty}(x)\ \text{for some}\ x\in D_\varepsilon\}.
\end{align*}
According to \eqref{n3-Kinfty-Onormal}, there exist $s_{\varepsilon,+},s_{\varepsilon,-}>0$ such that
$$
z_{\varepsilon,+}=(0,s_{\varepsilon,+},\varepsilon)\in {\rm cl}\,D_\varepsilon\mbox{ \ and \ }
z_{\varepsilon,-}=(0,-s_{\varepsilon,-},\varepsilon)\in {\rm cl}\,D_\varepsilon
$$
are the endpoints of the open arc $D_\varepsilon$ of $\partial K_\infty$, and there exist $\theta_{\varepsilon,+},\theta_{\varepsilon,-}\in(0,\pi)$ such that
$$
T^1((-\theta_{\varepsilon,-},\theta_{\varepsilon,+}))\subset \widetilde{F}_\varepsilon\subset T^1([-\theta_{\varepsilon,-},\theta_{\varepsilon,+}]),
$$
and hence $w_{\varepsilon,+}=(0,\sin \theta_{\varepsilon,+},-\cos \theta_{\varepsilon,+})$ and
$w_{\varepsilon,-}=(0,-\sin \theta_{\varepsilon,-},-\cos \theta_{\varepsilon,-})$ are exterior unit normals at $z_{\varepsilon,+}$ and $z_{\varepsilon,-}$.  We observe that there exist $s_{0,+},s_{0,-}\geq 0$ such that $e_3^\bot\cap\partial K_\infty={\rm conv}\{z_{0,+},z_{0,-}\}$ for $z_{0,+}=s_{0,+}e_2$ and $z_{0,-}=-s_{0,-}e_2$, and hence 
\begin{equation}
\label{n=3-limit-sepsilon-sinf}
\lim_{\varepsilon\to 0^+}s_{\varepsilon,+}=s_{0,+} \mbox{ \ and \ }\lim_{\varepsilon\to 0^+}s_{\varepsilon,-}=s_{0,-}.
\end{equation}
We claim that
\begin{equation}
\label{n=3-theta-over-epsilon}
\lim_{\varepsilon\to 0^+}\frac{\theta_{\varepsilon,+}}{\varepsilon}=\lim_{\varepsilon\to 0^+}\frac{\theta_{\varepsilon,-}}{\varepsilon}=\infty.
\end{equation}
To prove \eqref{n=3-theta-over-epsilon}, we only consider $\theta_{\varepsilon,+}$, the case of $\theta_{\varepsilon,-}$ can be handled similarly. We may assume that
$\theta_{\varepsilon,+}<\frac{\pi}2$ if $\varepsilon>0$ is small,
otherwise \eqref{n=3-theta-over-epsilon} readily holds.
 Since the inequality $\langle w_{\varepsilon,+},z_{\varepsilon,+}-z_{0,+}\rangle\geq 0$ is equivalent with $s_{\varepsilon,+}-s_{0,+}\geq \varepsilon/\tan \theta_{\varepsilon,+}$ by $\theta_{\varepsilon,+}<\frac{\pi}2$, we conclude \eqref{n=3-theta-over-epsilon} by \eqref{n=3-limit-sepsilon-sinf}.

We note that as $t_\infty e_3\in {\rm relint}\,K_\infty$ (cf. \eqref{n3-Kinfty-Onormal}), there exists $\theta_\infty\in(0,\pi)$ such that if $\varepsilon\in(0,t_\infty)$, then
\begin{equation}
\label{thetaeps-atmost-thetainf}
\theta_{\varepsilon,+},\theta_{\varepsilon,-}\leq \theta_\infty.
\end{equation}

We also consider two open subarcs $D_{\varepsilon,+}$ and $D_{\varepsilon,-}$ of $D_\varepsilon$ where $D_{\varepsilon,+}$ has $z_{\varepsilon,+}$ and $z_{0,+}$ as endpoints ,  and $D_{\varepsilon,-}$ has $z_{\varepsilon,-}$ and $z_{0,-}$ as endpoints. It follows that  for
\begin{align*}
F_{\varepsilon,+}=&T^2\left(\left(-\frac{\pi}4,\frac{\pi}4\right)\times \left(\frac{\theta_{\varepsilon,+}}4,\frac{\theta_{\varepsilon,+}}2\right) \right)\\
F_{\varepsilon,-}=&T^2\left(\left(-\frac{\pi}4,\frac{\pi}4\right)\times \left(-\frac{\theta_{\varepsilon,-}}2,-\frac{\theta_{\varepsilon,-}}4\right) \right),
\end{align*}
we have
$$
\nu^{-1}_{K_\infty}\left(F_{\varepsilon,+}\right)\subset{\rm cl}\,D_{\varepsilon,+} \mbox{ \ and \ }
\nu^{-1}_{K_\infty}\left(F_{\varepsilon,-}\right)\subset {\rm cl}\,D_{\varepsilon,-},
$$
and hence \eqref{tildeKm-to-Kinfty} yields that  if $y_m\in \nu^{-1}_{\widetilde{K}_m}(F_{\varepsilon,+})$ and $z_m\in \nu^{-1}_{\widetilde{K}_m}(F_{\varepsilon,-})$, then any accumulation point of $\{y_m\}$ is contained in ${\rm cl}\,D_{\varepsilon,+}$ and any accumulation point of $\{z_m\}$ is contained in ${\rm cl}\,D_{\varepsilon,-}$. We deduce the existence of a threshold $m_\varepsilon>1$ such that if $m\geq m_\varepsilon$, then
\begin{align}
x_2\geq -\varepsilon\mbox{ and } |x_3|\leq 2\varepsilon&\mbox{ \ for }(x_1,x_2,x_3)\in \nu^{-1}_{\widetilde{K}_m}\left(F_{\varepsilon,+}\right),\\
x_2\leq \varepsilon\mbox{ and } |x_3|\leq 2\varepsilon&\mbox{ \ for }(x_1,x_2,x_3)\in \nu^{-1}_{\widetilde{K}_m}\left(F_{\varepsilon,-}\right),
\end{align}
which in turn imply that  
\begin{align}
\label{n=3-nu-inv-F+Km}
x_2\geq -\varepsilon r_3(K_m)\mbox{ and } |x_3|\leq 2\varepsilon r_3(K_m)&\mbox{ \ for }(x_1,x_2,x_3)\in \nu^{-1}_{K_m}\left(F_{\varepsilon,+}\right),\\
\label{n=3-nu-inv-F-Km}
x_2\leq \varepsilon r_3(K_m)\mbox{ and } |x_3|\leq 2\varepsilon r_3(K_m)&\mbox{ \ for }(x_1,x_2,x_3)\in \nu^{-1}_{K_m}\left(F_{\varepsilon,-}\right).
\end{align}
In addition, as $\theta_\infty/2\in(0,\frac{\pi}2)$, combining \eqref{thetaeps-atmost-thetainf} and the definition of the parametrization $T^2$ of $S^2$ shows that 
\begin{equation}
\label{Feps-+-e3}
\langle \nu,-e_3\rangle\geq \frac{\cos \frac{\theta_\infty}2}{\sqrt{2}} \mbox{ \ for \ }\nu\in F_{\varepsilon,+}\cup F_{\varepsilon,-}.
\end{equation}

\begin{figure}
    \centering
    \includegraphics[width=1\linewidth]{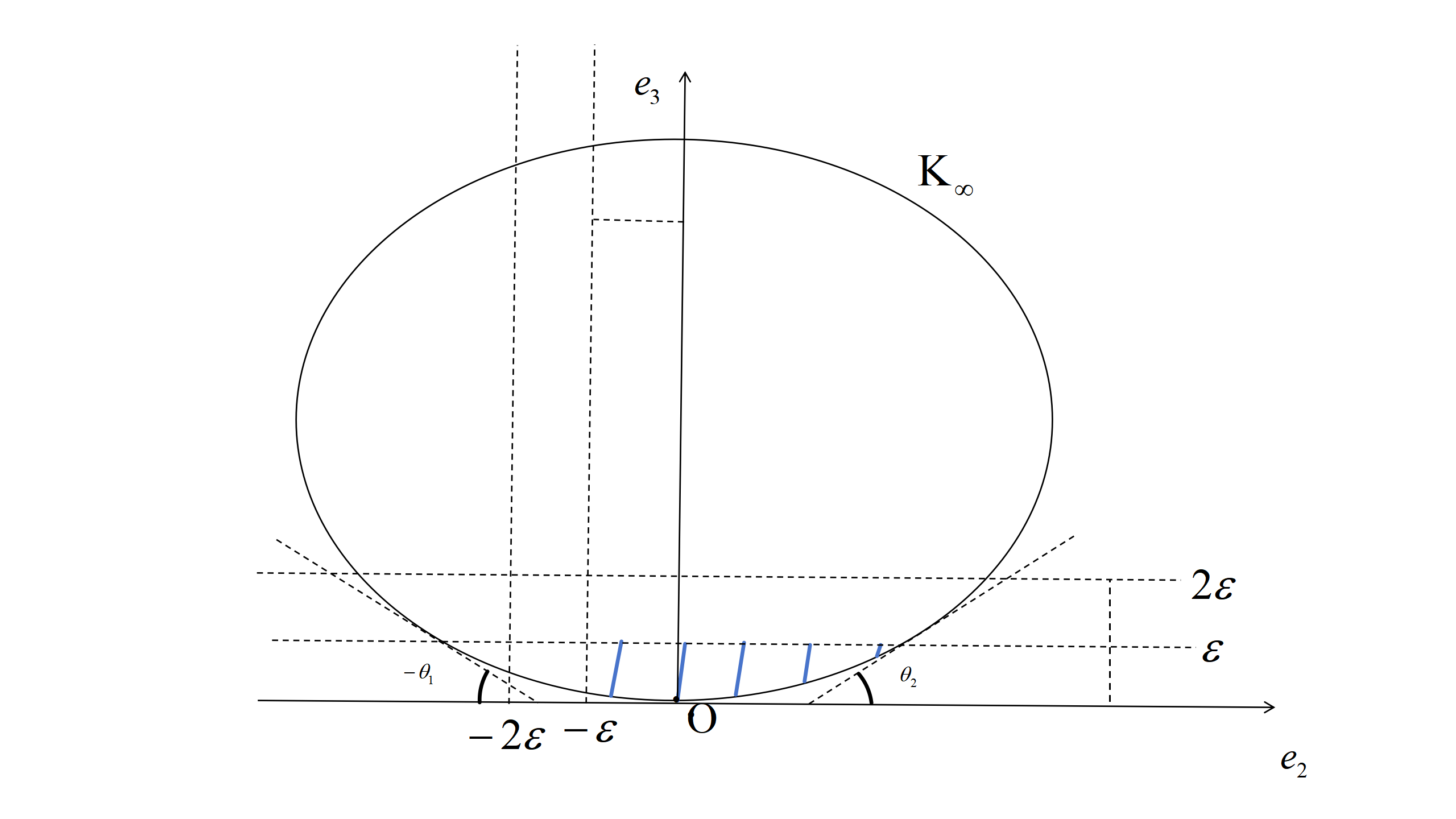}
    \caption{lemma \ref{q3 case II subcase ii.3 lemma}}
    \label{q3 case II subcase ii.3 lemma figure}
\end{figure}

 It follows from $f_m\geq \frac1\lambda$ and the definition of $F_{\varepsilon,+}$ and $F_{\varepsilon,-}$ that the density function $f_m$ of $\widetilde{C}_{p,q,K_m}$ satisfies
\begin{align}
\label{q3-angle-estimate+}
\int_{F_{\varepsilon,+}} f_{m}\,d\mathcal{H}^2\geq & c_0\theta_{\varepsilon,+}\\
\label{q3-angle-estimate-}
\int_{F_{\varepsilon,-}} f_{m} d\mathcal{H}^2\geq &c_0\theta_{\varepsilon,-}
\end{align}
where $c_0\in(0,1)$ depends on $\lambda$, and is independent of $\varepsilon$ and $m$.

Our next goal is to find some upper bound on $\int_{F_{\varepsilon,+}} f_{m} d\mathcal{H}^2$ or $\int_{F_{\varepsilon,-}} f_{m} d\mathcal{H}^2$. According to \eqref{nu-estimate} in Lemma~\ref{q3 case II subcase ii.1 lemma}
and \eqref{xi-estimate} in Lemma~\ref{q3 case II subcase ii.2 lemma}, for some
$\tilde{m}_\varepsilon\geq m_\varepsilon$, if $m>\tilde{m}_\varepsilon$, then
there exist unit vectors $\nu^{(m)}=(\nu_1^{(m)},\nu_2^{(m)},\nu_3^{(m)})\in S^2$ and $\xi^{(m)}=(\xi_1^{(m)},\xi_2^{(m)},\xi_3^{(m)})\in S^2$ such that 
\begin{equation}
\label{n=3-nu123-estimate0}
\nu_1^{(m)}\geq \frac{1}{2}, \quad |\nu_2^{(m)}| \lesssim \frac{r_1(K_m)}{r_3(K_m)}, \quad \text{and} \quad 
|\nu_3^{(m)}| \lesssim \frac{r_1(K_m)}{r_3(K_m)}, 
\end{equation}
\begin{equation}
\label{n=3-xi123-estimate0}
\xi_1^{(m)}\leq -\frac{1}{2}, \quad |\xi_2^{(m)}| \lesssim \frac{r_1(K_m)}{r_3(K_m)}, \quad \text{and} \quad |\nu_3^{(m)}| \lesssim \frac{r_1(K_m)}{r_3(K_m)}, 
\end{equation}
and $\nu^{(m)}$ and $\xi^{(m)}$ satisfy
\begin{align}
\label{n=3-nu-estimate0}
h_{K_m}(\nu^{(m)}) \leq &r_3(K_m)^{2-3(q-p)},\\
\label{n=3-xi-estimate0}
h_{K_m}(\xi^{(m)}) \leq & r_3(K_m)^{2-3(q-p)}.
\end{align}
We may also assume by \eqref{5.8 assumption} that if $m\geq \tilde{m}_\varepsilon$, then
\begin{equation}
\label{n=3-r1r3epsilon}
\frac{r_1(K_m)}{r_3(K_m)}<\varepsilon.
\end{equation}

Let $\Pi_3$ be the orthogonal projection from $\R^3$ to $e_3^\bot$, and hence $\Pi_3(x_1,x_2,x_3)=x_1e_1+x_2e_2$. Our key claim is that for $m>\tilde{m}_\varepsilon$, if $\frac{\nu_2^{(m)}}{\nu_1^{(m)}}+\frac{\xi_2^{(m)}}{|\xi_1^{(m)}|}\geq 0$, then
\begin{align}\label{n=3-Pi3-height-positive}
|x_1-y_1|\lesssim \varepsilon r_1(K_m) &\mbox{ for any } x_1e_1+z_2e_2\in \Pi_3\left(\nu_{K_m}^{-1} (F_{\varepsilon,+})\right) \\
\nonumber
 &\mbox{ and } y_1e_1+z_2e_2\in \Pi_3\left(\nu_{K_m}^{-1} (F_{\varepsilon,+})\right),\;x_1,y_1,z_2\in\R,
\end{align}
and if $\frac{\nu_2^{(m)}}{\nu_1^{(m)}}+\frac{\xi_2^{(m)}}{|\xi_1^{(m)}|}\leq 0$, then
\begin{align}\label{n=3-Pi3-height-negative}
|x_1-y_1|\lesssim \varepsilon r_1(K_m) &\mbox{ for any } x_1e_1+z_2e_2\in \Pi_3\left(\nu_{K_m}^{-1} (F_{\varepsilon,-})\right) \\
\nonumber
 &\mbox{ and } y_1e_1+z_2e_2\in \Pi_3\left(\nu_{K_m}^{-1} (F_{\varepsilon,-})\right),\;x_1,y_1,z_2\in\R.
\end{align}

First we assume that $\frac{\nu_2^{(m)}}{\nu_1^{(m)}}+\frac{\xi_2^{(m)}}{|\xi_1^{(m)}|}\geq 0$, and we may also assume that $x_1\geq y_1$ in \eqref{n=3-Pi3-height-positive}. We observe that there exist some 
$$
(x_1,z_2,x_3)\in \nu_{K_m}^{-1} (F_{\varepsilon,+}) \mbox{ \ and \ }(y_1,z_2,y_3)\in \nu_{K_m}^{-1} (F_{\varepsilon,+}),
$$ 
and  deduce from \eqref{n=3-nu-estimate0} and \eqref{n=3-xi-estimate0} that
\begin{align*}
\nu_1^{(m)}x_1+\nu_2^{(m)}z_2+\nu_3^{(m)}x_3\leq &h_{K_m}(\nu^{(m)}) \leq r_3(K_m)^{2-3(q-p)},\\
\xi_1^{(m)}y_1+\xi_2^{(m)}z_2+\xi_3^{(m)}y_3\leq &h_{K_m}(\xi^{(m)}) \leq r_3(K_m)^{2-3(q-p)}.
\end{align*}
Since $\nu_1^{(m)}\geq \frac{1}{2}$ by \eqref{n=3-nu123-estimate0} and $\xi_1^{(m)}\leq- \frac{1}{2}$ by \eqref{n=3-xi123-estimate0}, it follows that
\begin{align}
\nonumber
x_1-y_1\leq & r_3(K_m)^{2-3(q-p)}\left(\frac{1}{\nu_1^{(m)}}+\frac{1}{\left|\xi_1^{(m)}\right|}\right)-\\
\label{n=3-z2-place-to-estimate}
&-z_2\cdot \left(\frac{\nu_2^{(m)}}{\nu_1^{(m)}}+\frac{\xi_2^{(m)}}{\left|\xi_1^{(m)}\right|}\right)
-x_3\cdot \frac{\nu_3^{(m)}}{\nu_1^{(m)}}-y_3\cdot \frac{\xi_3^{(m)}}{\left|\xi_1^{(m)}\right|}.
\end{align}
Here $-z_2\leq \varepsilon r_3(K_m)$  and $|x_3|,|y_3|\leq 2\varepsilon r_3(K_m)$ in \eqref{n=3-z2-place-to-estimate} by \eqref{n=3-nu-inv-F+Km}, thus
$\frac{\nu_2^{(m)}}{\nu_1^{(m)}}+\frac{\xi_2^{(m)}}{|\xi_1^{(m)}|}\geq 0$, \eqref{n=3-nu123-estimate0} and \eqref{n=3-xi123-estimate0} yield that
$$
x_1-y_1\lesssim r_3(K_m)^{2-3(q-p)}+\varepsilon r_1(K_m).
$$
Now first we apply the estimate $r_3(K_m)^{-(q-p)} \lesssim \frac{r_1(K_m)}{r_3(K_m)}$ of Corollary~\ref{q3 estimate}, then \eqref{n=3-r1r3epsilon}, and finally $q>p+2$ and the estimate $r_3(K_m)\gtrsim 1$ of Corollary~\ref{q3 estimate} to deduce that
\begin{align*}
r_3(K_m)^{2-3(q-p)}\lesssim& r_3(K_m)^{2-q+p}\left(\frac{r_1(K_m)}{r_3(K_m)}\right)^2\\
\leq & r_3(K_m)^{1-q+p}\cdot \varepsilon r_1(K_m)\lesssim \varepsilon r_1(K_m),
\end{align*}
which in turn yields the claim \eqref{n=3-Pi3-height-positive} if $\frac{\nu_2^{(m)}}{\nu_1^{(m)}}+\frac{\xi_2^{(m)}}{|\xi_1^{(m)}|}\geq 0$.

On the other hand, if $\frac{\nu_2^{(m)}}{\nu_1^{(m)}}+\frac{\xi_2^{(m)}}{|\xi_1^{(m)}|}\leq 0$, then the argument is similar, only we use that $z_2\leq \varepsilon r_3(K_m)$  and $|x_3|,|y_3|\leq 2\varepsilon r_3(K_m)$ in \eqref{n=3-z2-place-to-estimate} by \eqref{n=3-nu-inv-F-Km} as $(x_1,z_2,x_3)\in \nu_{K_m}^{-1} (F_{\varepsilon,-})$, completing the proof of the claim also in this case.

As $h_{K_m}(e_2)\leq 6r_2(K_m)$ and $h_{K_m}(-e_2)\leq 6r_2(K_m)$, the claim and Fubini's theorem imply that if $m\geq \tilde{m}_\varepsilon$, then
\begin{align*}
\mathcal{H}^2\left(\Pi_3\left(\nu_{K_m}^{-1} (F_{\varepsilon,+})\right)\right)\lesssim &\varepsilon r_1(K_m)r_2(K_m)
\mbox{ \ or}\\ 
\mathcal{H}^2\left(\Pi_3\left(\nu_{K_m}^{-1} (F_{\varepsilon,-})\right)\right)\lesssim &\varepsilon r_1(K_m)r_2(K_m).
\end{align*}
As $\langle \nu_{K_m}(x),-e_3\rangle\geq \frac{\cos \frac{\theta_\infty}2}{\sqrt{2}}$ for any
$x\in \partial'K_m\cap \nu_{K_m}^{-1} (F_{\varepsilon,+})$ by \eqref{Feps-+-e3}, we conclude that
if $m\geq \tilde{m}_\varepsilon$, then
\begin{align}
\label{n=3-H2-nu-F+}
\mathcal{H}^2\left(\nu_{K_m}^{-1} (F_{\varepsilon,+})\right)\lesssim &\frac{\varepsilon}{\cos \frac{\theta_\infty}2}\cdot r_1(K_m)r_2(K_m)
\mbox{ \  or }\\ 
\label{n=3-H2-nu-F-}
\mathcal{H}^2\left(\nu_{K_m}^{-1} (F_{\varepsilon,-})\right)\lesssim &\frac{\varepsilon}{\cos \frac{\theta_\infty}2}\cdot r_1(K_m)r_2(K_m).
\end{align} 
If \eqref{n=3-H2-nu-F+} holds, then we deduce first using \eqref{CpqK-from-SK}, then $h_{K_m}(\nu)\leq  \|Dh_{K_m}(\nu)\|$ for $\nu\in S^2$ by \eqref{DhK-rhoK} and  $p<1$, next  $\|Dh_{K_m}\|\leq 6r_3$, $q>2+p$ and  \eqref{n=3-H2-nu-F+}, and finally the formula $r_1r_2 \approx r_3^{2-q+p}$ of Lemma~\ref{q3 main}
 to conclude that
\begin{align}
\nonumber
\int_{F_{\varepsilon,+}} f_{m}\,d\mathcal{H}^2=&\int_{F_{\varepsilon,+}} h_{K_m}^{1-p}\|Dh_{K_m}\|^{q-3}\,dS_{K_m}\leq \int_{F_{\varepsilon,+}}\|Dh_{K_m}\|^{q-p-2}\,dS_{K_m}\\
\label{n=3-intfm-upper-F+}
\lesssim &r_3(K_m)^{q-p-2}
\frac{\varepsilon}{\cos \frac{\theta_\infty}2}\cdot r_1(K_m)r_2(K_m)\leq C_+\varepsilon
\end{align}
where $C_+>1$ depends on $p$, $q$, $\lambda$, $\theta_\infty$ and the $C_1$ of \eqref{5.8 assumption}, and hence is independent of $m$ and $\varepsilon$.
Similarly, if \eqref{n=3-H2-nu-F-} holds, then we deduce the existence of $C_->1$ depending on $p$, $q$, $\lambda$, $\theta_\infty$ and the $C_1$ of \eqref{5.8 assumption}, and hence is independent of $m$ and $\varepsilon$, such that
\begin{equation}
\label{n=3-intfm-upper-F-}
\int_{F_{\varepsilon,-}} f_{m}\,d\mathcal{H}^2\leq C_-\varepsilon.
\end{equation}
Therefore,  if $m\geq \tilde{m}_\varepsilon$, then either \eqref{n=3-intfm-upper-F+} or \eqref{n=3-intfm-upper-F-} holds where $C_+$ and $C_-$ are independent of $m$ and $\varepsilon$. 

Finally, it follows from \eqref{n=3-theta-over-epsilon} that we can choose a positive $\varepsilon<t_\infty$ such that
$$
\min\left\{\frac{\theta_{\varepsilon,+}}{\varepsilon},\frac{\theta_{\varepsilon,-}}{\varepsilon}\right\}
>\max\left\{\frac{C_+}{c_0},\frac{C_-}{c_0}\right\}
$$
for the constants $c_0$ of \eqref{q3-angle-estimate+} and \eqref{q3-angle-estimate-}, $C_+$ of \eqref{n=3-intfm-upper-F+} and $C_-$ of \eqref{n=3-intfm-upper-F-}. Then considering an $m\geq \tilde{m}_\varepsilon$, either \eqref{n=3-intfm-upper-F+} contradicts \eqref{q3-angle-estimate+}, or
\eqref{n=3-intfm-upper-F-} contradicts \eqref{q3-angle-estimate-}. This contradiction
 completes the  proof of Lemma~\ref{q3 case II subcase ii.3 lemma}; namely, subcase (ii) of Case II does not occur.
\end{proof}

\subsection{Completing the argument for  Proposition~\ref{good-inequality-prop} and Theorem~\ref{n3-ppos--C0-estimate}}

\begin{proof}[Proof of Proposition~\ref{good-inequality-prop}, and in turn of Theorem~\ref{n3-ppos--C0-estimate}]
Combining Lemma~\ref{q3 case I subcase i lemma}, Lemma~\ref{q3 case I subcase ii.1 lemma}, Lemma~\ref{q3 case I subcaseii.2 lemma} (in Case I) and Lemma~\ref{q3 case II subcase i lemma} and Lemma~\ref{q3 case II subcase ii.3 lemma} (in Case II), we deduce that neither Case I nor Case II occurs. Thus, we obtain \eqref{good inequality} in Proposition~\ref{good-inequality-prop}. 
 
Now $r_1(K)\lesssim 1$ according to Corollary~\ref{q3 estimate}, and hence $r_3(K)\lesssim r_1(K)$ in Proposition~\ref{good-inequality-prop} yields ${\rm diam}\,K\leq 3r_3(K)\lesssim 1$. Also it is again Corollary~\ref{q3 estimate} says that $r_3(K)\gtrsim 1$, and hence $r_1(K)\gtrsim r_3(K)$ in Proposition~\ref{good-inequality-prop} implies $r_1(K)\gtrsim 1$. We deduce that $|K|\geq \frac{4\pi}3\,r_1(K)^3\gtrsim 1$, completing the proof of 
Theorem~\ref{n3-ppos--C0-estimate}.
\end{proof}

\section*{Acknowledgment}
\addcontentsline{toc}{section}{Acknowledgment}
Research of B\"or\"oczky was supported by Hungarian National Research and Innovation Office grant ADVANCED\_24 150613, and research of Chen was supported by National Key $R\&D$ program of China 2022YFA1005400, National Science Fund for Distinguished Young Scholars (No. 12225111).

\end{document}